\documentclass[reqno]{amsart}
\usepackage{graphics,array}
\usepackage{tikz}
\usetikzlibrary{matrix,arrows,patterns,cd,shapes}
\usepackage{amsmath,amsthm,mathtools,hyperref}
\usepackage{amsfonts}
\usepackage{amssymb}

 \newtheorem{theorem}{Theorem}[section]
 \newtheorem{corollary}[theorem]{Corollary}
 \newtheorem{prop}[theorem]{Proposition}
 
 \newtheorem{lemma}[theorem]{Lemma}
 
 \theoremstyle{definition}
 \newtheorem{example}[theorem]{Example}
 
 \newtheorem{problem}[theorem]{Problem}
 \newtheorem{definition}[theorem]{Definition}
  \newtheorem{remark}[theorem]{Remark}

 \theoremstyle{remark}
 \newtheorem{rem}[theorem]{Remark}

\newcommand{\flip}[1]{\mathsf{flip}(#1)}
\newcommand{\asc}{\mathsf{asc}}

\newcommand{\bij}{f}
\newcommand{\bam}{\bij_{\mathsf{AM}}}  % (b)ijection (a)scent_sequences to (m)atrices
\newcommand{\bma}{\bij_{\mathsf{MA}}}
\newcommand{\bap}{\bij_{\mathsf{AP}}}
\newcommand{\bas}{\bij_{\mathsf{AS}}} % S for {s}ymmetric group
\newcommand{\bpa}{\bij_{\mathsf{PA}}}
\newcommand{\bmp}{\bij_{\mathsf{MP}}}
\newcommand{\bpm}{\bij_{\mathsf{PM}}}
\newcommand{\bps}{\bij_{\mathsf{PS}}}

\newcommand{\am}[1]{$\mathsf{AM#1}$}
\newcommand{\ma}[1]{$\mathsf{MA#1}$}
\newcommand{\ap}[1]{$\mathsf{AP#1}$}
\newcommand{\pa}[1]{$\mathsf{PA#1}$}

\newcommand{\Asc}{\mathsf{Asc}}
\newcommand{\RAsc}{\mathsf{RAsc}}
\newcommand{\CAsc}{\mathsf{CAsc}}
\newcommand{\Posets}{\mathsf{Posets}}
\newcommand{\RPosets}{\mathsf{RPosets}}
\newcommand{\CPosets}{\mathsf{CPosets}}
\newcommand{\Matrices}{\mathsf{Matrices}}
\newcommand{\RMatrices}{\mathsf{RMatrices}}
\newcommand{\CMatrices}{\mathsf{CMatrices}}
\newcommand{\Perms}{\mathsf{Perms}} % Other option: $\sn(2|3\overline{1})$
\newcommand{\RPerms}{\mathsf{RPerms}} % Other option: $\sn(3\overline{1}52\overline{4})$
\newcommand{\CPerms}{\mathsf{CPerms}}  %Other option:  $\sn(231)$
\newcommand{\sn}{\mathcal{S}_n}
\newcommand{\RGF}{\mathsf{RGF}}
\newcommand{\Structure}{\mathsf{Struct}}
\newcommand{\reverse}{\mathrm{rev}}
\renewcommand{\complement}{\mathrm{comp}}
\newcommand{\D}{\mathsf{pan}}
\newcommand{\R}{\mathsf{R}}
\newcommand{\C}{\mathsf{C}}

\newcommand{\level}{\ell}
\newcommand{\eell}{{l}}
\newcommand{\minlevel}{\mathrm{\ell}^{\star}}
\newcommand{\remove}{\mathsf{remove}}

\newcommand{\mindex}{\mathsf{mindex}}
\newcommand{\reduce}{\mathsf{reduce}}
\newcommand{\rowsum}{\mathsf{rowsum}}
\renewcommand{\max}{\mathsf{max}}

\newcommand{\A}{\mathbf{a}}
\renewcommand{\a}{a}

\numberwithin{equation}{section}
\numberwithin{figure}{section}
\renewcommand{\theenumi}{\alph{enumi}}

\definecolor{darkgreen}{rgb}{0,0.3,0}

\title[Bijections for ascent sequences]{Refining the bijections among ascent sequences, (2+2)-free posets, integer matrices and pattern-avoiding permutations}

\author{Mark Dukes}%
\address{UCD School of Mathematics and Statistics, University College Dublin, Dublin 4, Ireland} \email{\href{mailto:mark.dukes@ucd.ie}{mark.dukes@ucd.ie}}
\urladdr{\href{http://maths.ucd.ie/~dukes/}{maths.ucd.ie/~dukes}}

\author{Peter R.\,W. McNamara}
\address{Department of Mathematics, Bucknell University, Lewisburg, PA 17837, USA}
\email{\href{mailto:peter.mcnamara@bucknell.edu}{peter.mcnamara@bucknell.edu}}
\urladdr{\href{http://www.facstaff.bucknell.edu/pm040}{www.facstaff.bucknell.edu/pm040}}

\begin{document}

\begin{abstract}
The combined work of Bousquet-M\'elou, Claesson, Dukes, Jel\'inek, Kitaev, Kubitzke and Parviainen has resulted in non-trivial bijections among ascent sequences, (2+2)-free posets, upper-triangular integer matrices, and pattern-avoiding permutations.  To probe the finer behavior of these bijections, we study two types of restrictions on ascent sequences.  These restrictions are motivated by our results that their images under the bijections are natural and combinatorially significant.  In addition, for one restriction, we are able to determine the effect of poset duality on the corresponding ascent sequences, matrices and permutations, thereby answering a question of the first author and Parviainen in this case.  The second restriction should appeal to \linebreak 
Catalaniacs.
\end{abstract}

\maketitle

%%%%%%%%%%%%%%%%%%%%%%%%%%%%%%%
%%%%%%%%%%%%%%%%%%%%%%%%%%%%%%%
\section{Introduction}
%%%%%%%%%%%%%%%%%%%%%%%%%%%%%%%
%%%%%%%%%%%%%%%%%%%%%%%%%%%%%%%

In the last decade, an interesting collection of results has emerged from the study of (2+2)-free posets, or interval orders as they are also known,
and their connection to permutations avoiding a non-standard permutation pattern of length three.
The starting point for this story was the introduction in Bousquet-M\'elou et al.~\cite{bcdk} of a new type of permutation pattern that the authors termed a {\it{bivincular pattern}}.
In that paper it was proven that length-$n$ permutations avoiding the bivincular pattern
$2|3\overline{1}$ were in one-to-one correspondence with unlabelled (2+2)-free posets on $n$ elements.
This was shown by encoding both structures as an integer sequence of length $n$ that has come to be known as an {\it{ascent sequence}}.  Via ascent sequences, Bousquet-M\'elou et al.\ were able to solve the long-standing open problem of enumerating unlabelled (2+2)-free posets.

In \cite{dp}, these ascent sequences were shown to uniquely encode another set of objects:
all square upper-triangular matrices of non-negative integers whose entries sum to $n$ which have neither rows nor columns consisting of only zeros.  Such matrices are known as \emph{Fishburn matrices} since they were introduced by Fishburn \cite{Fis}; these matrices and the three previously mentioned classes are enumerated by the Fishburn numbers \cite[\href{https://oeis.org/A022493}{A022493}]{oeis}.

These initial two papers linking (bijectively) four different discrete objects led to a series of papers that studied these bijections and built upon the correspondences.
From the enumerative viewpoint, Dukes et al.~\cite{dkrs} and Kitaev \& Remmel~\cite{kr} considered these objects according to several statistics (such as the number of minimal/maximal elements) and presented multivariate generating functions for these statistics.
In Claesson et al.~\cite{cdkub}, the bijections from the original papers were lifted to achieve bijections between
labelled (2+2)-free posets, upper-triangular matrices whose entries partition a set, and a form of coloured ascent sequences.
This lift of the correspondences to the labelled setting was used to give a bijection from unlabelled (2+2)-free posets to Fishburn matrices~\cite{djk}, which is equivalent to the definition of the \emph{characteristic matrix} in \cite[\S 2.3]{Fis}).
%that would not have been possible using only the unlabelled structures 

In another direction, and more recently, it has emerged that refinements of these correspondences have equally compelling stories to tell.
Duncan \& Steingr\'imsson~\cite{ds} studied pattern avoidance in ascent sequences and established bijections between pattern avoiding ascent sequences and other combinatorial objects such as set partitions and objects enumerated by the Catalan and Narayana numbers.
Jel\'inek~\cite{JelSelfdual} presented a new method to derive formulas for the generating functions of interval orders.
The method generalised the results of~\cite{dkrs,kr} and also allowed the enumeration of self-dual interval orders with respect to several statistics.
Using his newly derived generating function formulas, Jel\'inek proved a bijective relationship between self-dual interval orders and upper-triangular matrices having no zero rows~\cite{JelSelfdual}.  
Andrews \& Jel\'inek~\cite{AndJel} built on Jel\'inek's work and proved several power series identities involving the refined generating functions for interval orders and self-dual interval orders.
Keller and Young~\cite{ky} considered the difficult question of determining which ascent sequences map to semiorders; also known as unit interval orders, semiorders are posets that are both (2+2) and (3+1)-free. See also \cite{kr} for a consideration of this question.

% Our work
The present paper adds to this body of work by analyzing two types of restrictions on ascent sequences.
One motivation for these restrictions is that their images through the bijections of~\cite{bcdk,djk,dp} are combinatorially significant in the subsets they identify,
e.g. series-parallel posets and 231-avoiding permutations.
Moreover, the analysis of the images of these ascent sequences allows us to prove results about duals of each of the structures,
thus going some way in answering an open problem of Dukes \& Parviainen~\cite{dp}.

The first type of restriction we study (in Section 3) begins with a restriction on the types of ascents one may have in an ascent sequence.  
In particular, when the bijection of~\cite{bcdk} recursively builds a (2+2)-free poset from an ascent sequence, there are some ascents that cause complicated and unnatural modifications to the poset, while the bijection treats all other ascents in a very natural way.  Our first restriction is to those ascent sequences that contain only these ascents that result in this latter natural behaviour.  A motivation for this restriction is that this good behaviour carries through to the general framework of bijections.  Indeed, the images of these new \emph{restricted ascent sequences} $\RAsc$ through the bijections given in~\cite{bcdk,djk,dp} are proven to be simple restrictions:
the subset $\RMatrices$ of the matrices from~\cite{dp} having only positive diagonal entries,
the subset $\RPosets$ of (2+2)-free posets which have a chain 
of the maximal possible length, and the set $\RPerms$ of permutations avoiding the barred pattern $3 \overline{1}5 2\overline{4}$. This set $\RPerms$ was already identified in~\cite{bcdk} in the context of modified ascent sequences.  See Figure~\ref{fig:summary} for a diagram outlining our sets and maps of interest.
 
In Section 4 we give a partial solution to an open problem of Dukes \& Parviainen~\cite{dp} by addressing the topic of structural duality.
The dual $P^*$ of a (2+2)-free poset $P$ is also a (2+2)-free poset.
This observation prompts the question as to whether one can derive the ascent sequence (resp.\ permutation) corresponding to $P^*$ from the ascent sequence (resp.\ permutation)  corresponding to $P$.  This question seems intractable in general because of the complicated map between some ascent sequences and posets as mentioned in the previous paragraph.  However, consistent with our motivation for restricting to better-behaved sets, we can answer this duality question completely for all posets in $\mathsf{RPosets}$, which we do in Section 4.

In Section 5, we consider the Catalan family $\CAsc$ of 101-avoiding ascent sequences studied in~\cite{ds}, and investigate their images under the bijections of~\cite{bcdk,djk,dp}.  (The restricted ascent sequences of \cite{kr} are also enumerated by Catalan numbers but are different from $\CAsc$.)  The results are perhaps even nicer than the $\R$-families and are shown in Figure~\ref{fig:summary}.  The posets that arise are the series-parallel interval orders, i.e., those that are both (2+2)-free and N-free.  This class of posets appears in \cite{dfpr10, dfpr12}, while series-parallel posets in general are widespread in the literature, partially because their recursive structure permits many polynomial-time algorithms (see, for example, \cite{Gor} and the references therein).  The matrices and permutations which correspond to these ascent sequences are those matrices from $\RMatrices$ that are termed \emph{SE-free} in \cite{JelFishburn}, and 231-avoiding permutations, respectively.

We conclude in Section~\ref{sec:conclusion} with some open questions.

\begin{figure}
\begin{tikzpicture}[scale=0.5, rounded corners = 5pt]
\begin{scope}
%\draw (0,7.5) node {$\Posets$};
\draw (0,6.05) node {$\Posets$:\,(2+2)-free posets $P$};
\draw (0,4.5) node {$\RPosets$: $P$ has a};
\draw (0,3.6) node {chain of length $\level(P)$};
\draw (0,2.2) node {$\CPosets$:};
\draw (0,1.3) node {series-parallel $P$};
\draw (-5,0) rectangle (5,7);
\draw (-4,0.2) rectangle (4,5);
\draw (-3,0.4) rectangle (3,3);
\end{scope}
\begin{scope}[xshift = 90ex]
%\draw (0,7.5) node {$\Asc$};
\draw (0,6.05) node {$\Asc$: ascent sequences $\A$};
\draw (0,4.05) node {$\RAsc$: self-modified $\A$};
\draw (0,2.2) node {$\CAsc$:};
\draw (0,1.3) node {$101$-avoiding $\A$};
\draw (-5,0) rectangle (5,7);
\draw (-4,0.2) rectangle (4,5);
\draw (-3,0.4) rectangle (3,3);
\end{scope}
\begin{scope}[yshift = -75ex]
%\draw (0,7.5) node {$\Matrices$};
\draw (0,6.5) node {$\Matrices$: upper-$\triangle$ $\mathbb{N}$-matrices};
\draw (0,5.6) node {$M$ w/o empty rows or columns};
\draw (0,4.5) node {$\RMatrices$: $M$ has only};
\draw (0,3.6) node {positive diagonal entries};
\draw (0,2.2) node {$\CMatrices$:};
\draw (0,1.3) node {$M$ is SE-free};
\draw (-5,0) rectangle (5,7);
\draw (-4,0.2) rectangle (4,5);
\draw (-3,0.4) rectangle (3,3);
\end{scope}
\begin{scope}[xshift = 90ex, yshift = -75ex]
%\draw (0,7.5) node {$\Perms$};
\draw (0,6.05) node {$\Perms$: $\sn(2|3\overline{1})$};
\draw (0,4.05) node {$\RPerms$: $\sn(3\overline{1}52\overline{4})$};
\draw (0,1.7) node {$\CPerms$: $\sn(231)$};
\draw (-5,0) rectangle (5,7);
\draw (-4,0.2) rectangle (4,5);
\draw (-3,0.4) rectangle (3,3);
\end{scope}
\draw[->] (-1,-1) -- (-1,-3);
\draw (-1.7,-2) node {$\bpm$};
\draw[<-] (1,-1) -- (1,-3);
\draw (1.7,-2) node {$\bmp$};
\draw[->] (5.8,4.5) -- (7.8,4.5);
\draw (6.8,5) node {$\bpa$};
\draw[<-] (5.8,2.5) -- (7.8,2.5);
\draw (6.8, 2) node {$\bap$};
\draw[->] (13.7,-1) -- (13.7,-3);
\draw (13.0,-2) node {$\bas$};
\draw[->] (5.2,-2.8) -- (7.2,-0.8);
\draw (5.7, -1.3) node {$\bma$};
\draw[<-] (6.2,-3.7) -- (8.2,-1.7);
\draw (7.75, -3.05) node {$\bam$};

\end{tikzpicture}
\caption{A diagrammatic summary of the sets and bijections of interest.}
\label{fig:summary}
\end{figure}
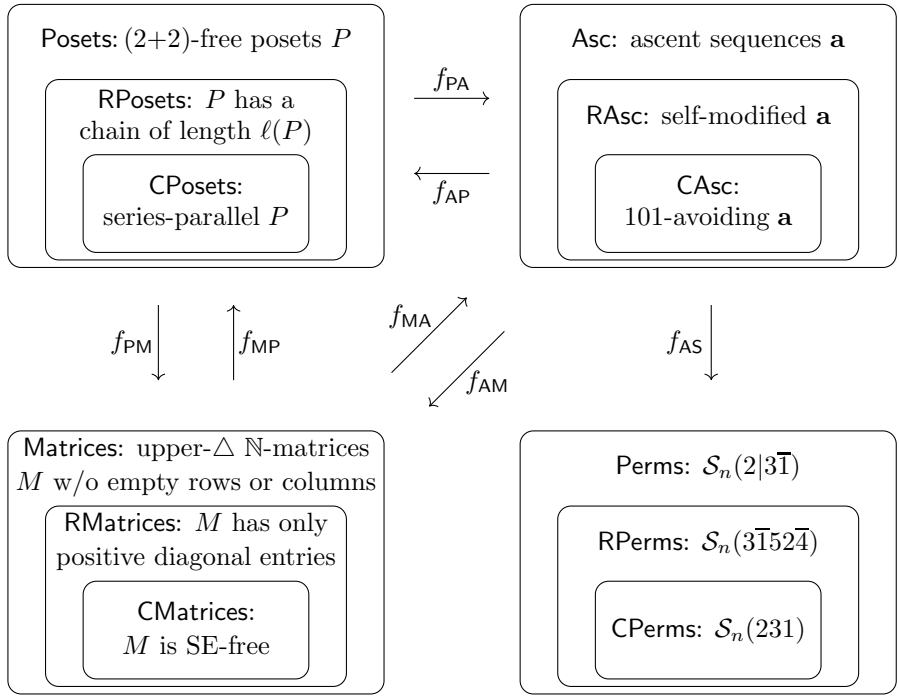

%%%%%%%%%%%%%%%%%%%%%%%%%%%%%%%
%%%%%%%%%%%%%%%%%%%%%%%%%%%%%%%
\section{Preliminaries}
%%%%%%%%%%%%%%%%%%%%%%%%%%%%%%%
%%%%%%%%%%%%%%%%%%%%%%%%%%%%%%%

While the intimate connections between four different types of objects are certainly a strength of this area of study, the drawback for our present purposes is there is a considerable amount of background that needs to be introduced, including all four classical sets and many of the bijections among them.  Our use of the words ``classical'' refers to the full sets considered by most of the papers mentioned in the Introduction, and as shown by the largest boxes in Figure~\ref{fig:summary}:  ascent sequences, (2+2)-free posets, upper-triangular matrices with non-negative integer entries having neither rows nor columns of all zeros, and permutations avoiding $2|3\overline{1}$.

%%%%%%%%%%%%%%%%%%%%%%%%%%%%%%%
\subsection{The classical sets}\label{sub:classicalsets}
%%%%%%%%%%%%%%%%%%%%%%%%%%%%%%%

An \emph{ascent sequence} is a sequence $\A = (\a_1, \ldots,\a_n)$ of non-negative integers such that $\a_1=0$, and for all $i$ with $1<i \leq n$ we have $\a_i \leq \asc(\a_1,\ldots,\a_{i-1}) +1$, where $\asc(\a_1,\ldots,\a_k)$ denotes the number of ascents in the sequence $(\a_1,\ldots,\a_k)$.  For example $(0,1,0,1,3)$ is an ascent sequence whereas $(0,1,0,2,4)$ is not.  Let $\Asc_n$ be the set of all ascent sequences of length $n$, and let $\Asc$ denote the union of these sets over all $n$, with the same convention applying to all the notation below when the subscript $n$ is dropped.

Let $\Posets_n$ be the set of (2+2)-free posets on $n$ elements, meaning posets that have no induced subposet isomorphic to a disjoint union of two 2-element chains.  We will be interested in a different defining property of (2+2)-free posets, as we now describe.  Let $P=(P,\preceq)$ be a poset with $n$ elements.  Given $x \in P$, the set $D(x)=\{y \in P : y \prec_P x\}$ is called the \emph{strict downset} of $x$.  A fact described as ``well known'' in~\cite{Bog} and which is easy to check is that a poset is (2+2)-free if and only if the set of strict downsets of elements of $P$ can be linearly ordered by inclusion.  We let $\level(P)$ denote the number of distinct nonempty such downsets, so that $D(P)=(D_0,\ldots,D_{\level(P)})$ is the sequence of downsets of $P$ linearly ordered by inclusion.  In other words $\emptyset = D_0 \subsetneq D_1 \subsetneq \ldots \subsetneq D_{\level(P)}$.  For example, for the poset $P$ in the top left of Figure~\ref{fig:examples}, we have
\[
D(P) = (\emptyset, \{p_1,p_2\}, \{p_1, p_2, p_5\}, \{p_1, p_2, p_3, p_5\}).
\]
(Note that while this example $P$ is labelled for the purposes of the explanation, the elements of $\Posets$ are unlabelled.)
We will call $D_i$ \emph{level} $i$ of $P$, and an element $x \in P$ with $D(x)=D_i$ for some $i$ will be said to \emph{lie at level} $i$ of $P$.
Let $L_i=L_i(P)$ be the set of elements lying at level $i$ of $P$ and set $L(P)=(L_0,\ldots,L_{\level(P)})$.
Again for $P$ from Figure~\ref{fig:examples}, 
\[
L(P) = (\{p_1, p_2, p_5\}, \{p_6\}, \{p_3\}, \{p_4\}).
\]
%We will refer to the sequence $L(P)$ as the level sequence of $P$.
%

Let $\Matrices_n$ be the set of all upper-triangular matrices whose entries are all non-negative integers such that there is neither a row nor a column containing only zeros, and whose sum of all entries is $n$.  Observe that the dimension $d$ of an element of $\Matrices_n$ satisfies $1 \leq d \leq n$.
% Throughout, the entries of the matrix $M$ will be denoted $m_{ij}$ or $m_{i,j}$ but we use $M_{ij}$ when the entries of the matrix are sets.

A sequence $\A = (\a_1, \ldots, \a_r)$ of non-negative integers is said to \emph{contain} a sequence $\mathbf{b} =(b_1, \ldots, b_s)$ \emph{as a pattern} if there exists a subsequence of $\A$ of length $s$ whose elements are in the same relative order as those of $\mathbf{b}$.  We say $\A$ is \emph{$\mathbf{b}$-avoiding} if it does not contain $\mathbf{b}$.  For example, $(0,2,1,3,1,0,2)$, which we write as 0213102 for short, contains the pattern $0101$ because of its subsequence $0202$, but avoids the pattern $1010$.  

When considering pattern-avoidance in sequences that are permutations, we allow for a more general notion of pattern: a permutation $\pi = \pi_1\ldots \pi_n$ is said to \emph{contain} the pattern $2|3\overline{1}$ if there exists an occurrence $\pi_i \pi_j \pi_k$ of 231 in $\pi$ with the additional conditions that $j=i+1$ and $\pi_i = \pi_k+1$.  For example, $32541$ contains $2|3\overline{1}$ because of the occurrence $251$ of $231$, whereas $31452$ avoids $2|3\overline{1}$ even though it has three occurrences of the classical pattern $231$.  The pattern $2|3\overline{1}$ is an example of a \emph{bivincular pattern} as introduced in~\cite{bcdk} since it puts conditions on both the entries and positions of an occurrence.  As usual, we let $\sn(2|3\overline{1})$ denote the set of permutations of length $n$ that avoid $2|3\overline{1}$, and this is exactly our set $\Perms_n$.

%%%%%%%%%%%%%%%%%%%%%%%%%%%%%%%
\subsection{The bijections}\label{sub:bijections}
%%%%%%%%%%%%%%%%%%%%%%%%%%%%%%%

In this subsection, we gather the classical bijections from the literature~\cite{bcdk,djk,dp} that we need.  We refer the reader to these references for the proofs of bijectivity and our claims that particular pairs of maps are inverses.  Due to its length, the reader may prefer to skip this subsection and instead refer back to it as a reference.  

To denote the bijections from~\cite{bcdk,djk,dp}, we will use labels according to their domain and codomain, but rather than use the labels $\Asc$, $\Posets$, $\Matrices$ and $\Perms$, we use the single-letter subscripts $\mathsf{A}$, $\mathsf{P}$, $\mathsf{M}$ and $\mathsf{S}$ (``$\mathsf{S}$'' for ``symmetric group'').  For example, $\bap$ denotes the bijection of~\cite{bcdk} from ascent sequences to (2+2)-free posets.

If $a<b$ are integers, we use the notation $[a,b]$ for the set $\{a, \ldots,b\}$ and $[a,b)$ for the set $\{a,\ldots,b-1\}$, etc.  

For each of our bijections, we will refer to the example in Figure~\ref{fig:examples}.

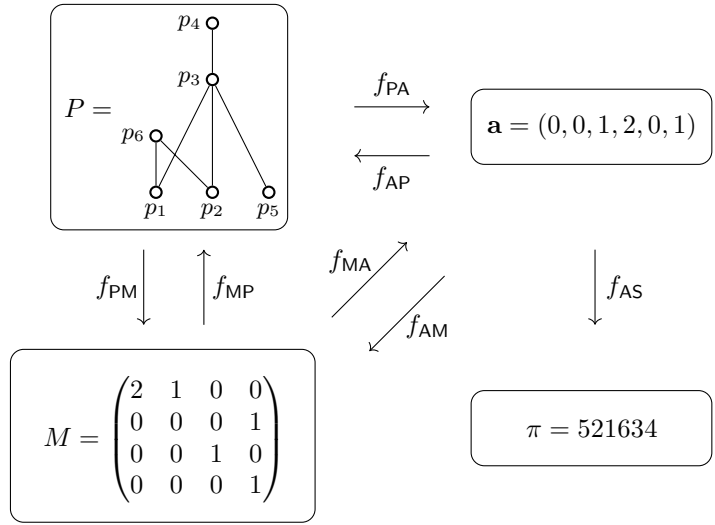
\begin{figure}
\begin{tikzpicture}[scale=0.5, rounded corners = 5pt]
\begin{scope}[yshift=15, circle, xshift=15, inner sep=1.5pt];
\draw (-2.8,-1.0) rectangle (3.5,5);
\draw (-1.8,2.25) node {$P=$};
\draw[thick] (0,0) node[draw] (1) {$$};
\draw[thick] (1.5,0) node[draw] (2) {$$};
\draw[thick] (3,0) node[draw] (5) {$$};
\draw[thick] (1.5,3) node[draw] (3) {$$};
\draw[thick] (1.5,4.5) node[draw] (4) {$$};
\draw[thick] (0,1.5) node[draw] (6) {$$};
\draw (1) -- (6) -- (2) -- (3) -- (1);
\draw (4) -- (3) -- (5);
\draw (0,-0.5) node {\small$p_1$};
\draw (1.5,-0.5) node {\small$p_2$};
\draw (3,-0.5) node {\small$p_5$};
\draw (-0.6,1.5) node  {\small$p_6$};
\draw (0.9,3) node  {\small$p_3$};
\draw (0.9,4.5) node  {\small$p_4$};
\end{scope}
\begin{scope}[xshift = 80ex, yshift=15ex]
\draw (0,0) node {$\A = (0,0,1,2,0,1)$};
\draw (-3.2,-1) rectangle (3.2,1);
\end{scope}
\begin{scope}[yshift = -40ex, xshift=5ex]
\draw (-4.1,-2.2) rectangle (4,2.4);
\draw (0,0) node {$M = 
\begin{pmatrix} 2 & 1 & 0 & 0 \\ 0 & 0 & 0 & 1 \\ 0 & 0 & 1 & 0 \\ 0 & 0 & 0 & 1 \end{pmatrix}$};
\end{scope}
\begin{scope}[xshift = 80ex, yshift = -38ex]
\draw (0,0) node {$\pi = 521634$};
\draw (-3.2,-1) rectangle (3.2,1);
\end{scope}
\draw[->] (0.2,-1) -- (0.2,-3);
\draw (-0.5,-2) node {$\bpm$};
\draw[<-] (1.8,-1) -- (1.8,-3);
\draw (2.6,-2) node {$\bmp$};
\draw[->] (5.8,2.8) -- (7.8,2.8);
\draw (6.8,3.4) node {$\bpa$};
\draw[<-] (5.8,1.5) -- (7.8,1.5);
\draw (6.8, 0.9) node {$\bap$};
\draw[->] (12.2,-1) -- (12.2,-3);
\draw (13.0,-2) node {$\bas$};
\draw[->] (5.2,-2.8) -- (7.2,-0.8);
\draw (5.7, -1.3) node {$\bma$};
\draw[<-] (6.2,-3.7) -- (8.2,-1.7);
\draw (7.75, -3.05) node {$\bam$};

\end{tikzpicture}
\caption{The examples used when defining the bijections in Subsection~\ref{sub:bijections}.}
\label{fig:examples}
\end{figure}

\subsubsection{Ascent sequences and matrices}\label{sub:bam}
Let $\A = (\a_1, \ldots, \a_n)$ be an ascent sequence and define the truncated sequence $\A^{(k)}=(a_1,\ldots,a_k)$.
If $M$ is a $d \times d$ matrix, then we write $\dim(M) = d$.  Let $\mindex(M)$ be the lowest index of a row whose rightmost entry is non-zero.  For example, $\mindex(M)=2$ for $M$ from Figure~\ref{fig:examples}.  The following map first appears in~\cite{dp} where it is shown to be a bijection to $\Matrices_n$.

\begin{definition}\label{def:bam}
Given $\A = (\a_1, \ldots, \a_n) \in \Asc_n$, we define $\bam(\A)$ recursively.  First, $\bam(\A^{(1)}) = (1)$, a $1\times 1$ matrix.  Supposing $M^{(k)} = \bam(\A^{(k)})$ for some $k \in [1,n)$, we have the following three cases for defining $M^{(k+1)} = \bam(\A^{(k+1)})$.
\begin{enumerate}
\item[\am1]  If $\a_{k+1} \in [0,\mindex(M^{(k)}))$ then let $M^{(k+1)}$ be the matrix $M^{(k)}$ with the entry at position $(\a_{k+1}+1,\dim(M^{(k)}))$ increased by one.
\item[\am2] If $\a_{k+1}=\dim(M^{(k)})$ then let $M^{(k+1)}$ be the result of appending a new row and column of zeros to the matrix $M^{(k)}$, and inserting a one into the new diagonal entry.
\item[\am3] If $\a_{k+1} \in [\mindex(M^{(k)}),\dim(M^{(k)}))$ then let $M^{(k+1)}$ be the outcome of the following: in $M^{(k)}$, insert a new (empty) row between rows $\a_{k+1}$ and $1+\a_{k+1}$, and insert a new (empty) column between columns $\a_{k+1}$ and $1+\a_{k+1}$. Let the new row be filled with all zeros except let the rightmost entry be 1.  Move all the entries in the rightmost column above where it was prised apart to the left to the new empty spaces, and fill their former positions with zeros. Finally let all other entries in the new column be zero. 
\end{enumerate}
Then $\bam(\A) = \bam(\A^{(n)}) = M^{(n)}$.  
\end{definition}

\begin{example}\label{exa:bam}
Let $\A = (0,0,1,2,0,1)$ as in Figure~\ref{fig:examples}. 
\begin{itemize}
\item  $\bam(\A^{(1)})=(1)=M^{(1)}$. 
\item Since $\a_{2} \in [0,\mindex(M^{(1)})=1)$, \am1 applies and $M^{(2)}=(2)$.
\item Since $\a_{3} =\dim(M^{(2)})$, \am2 applies and $M^{(3)}=\left(\begin{smallmatrix} 2 & 0 \\ 0 & 1 \end{smallmatrix}\right)$.
\item Since $\a_{4} =\dim(M^{(3)})$, \am2 applies and $M^{(4)}=\left(\begin{smallmatrix} 2 & 0 &0\\ 0 & 1 & 0 \\ 0&0&1 \end{smallmatrix}\right)$.
\item Since $\a_{5} \in [0,\mindex(M^{(4)})=3)$, \am1 applies and $M^{(5)}=\left(\begin{smallmatrix} 2 & 0 &1\\ 0 & 1 & 0 \\ 0&0&1 \end{smallmatrix}\right)$.
\item Since $\a_{6} \in [\mindex(M^{(5)})=1,\dim(M^{(5)})=3)$, \am3 applies and we prise the matrix $M^{(5)}$ apart between rows 1 and 2, and columns 1 and 2 to get $\left(\begin{smallmatrix} 2&\ &0&1 \\ &&& \\ 0&&1&0 \\ 0&&0&1 \end{smallmatrix}\right)$.  We fill the new second row with zeros except for a 1 in column 4.  We move the single entry 1 in column 4 above the new row into the new second column, preserving this entry's row index,  and put zeros in its former position and in the rest of the new column.
We get $M^{(6)}=\left(\begin{smallmatrix} 2&1&0&0 \\ 0&0&0&1 \\ 0&0&1&0 \\ 0&0&0&1 \end{smallmatrix}\right) = \bam(\A)$.
\end{itemize}
\end{example}

As is clear in the above definition, \am3 is by far the most involved, and this will be a feature of \ap3 too.  When \am3 is used in a recursive step, the results of the bijection quickly become intractable.  The point of the $\R$-families is that they are defined exactly so that \am3 and \ap3 are never invoked, resulting in more manageable analysis.  

We could define $\bma$ as just the inverse of the bijection $\bam$ but it will be useful to write an explicit definition following~\cite{dp}.  
%We actually won't need the \ma3 part of the definition for our results but it is necessary for giving a complete inverse to $\bam$. 
Example~\ref{exa:bam} in reverse serves as an example of Definition~\ref{def:bma}.  Let $\rowsum_i(M)$ denote the sum of the entries in row $i$ of $M$.

\begin{definition}\label{def:bma}
Given $M \in \Matrices_n$, define $\reduce(M) \in \Matrices_{n-1}$ in the following fashion.
\begin{enumerate}
\item[\ma1] If $\rowsum_{\mindex(M)}(M) > 1$ then let $\reduce(M)$ equal $M$ with the value at position $(\mindex(M), \dim(M))$ reduced by 1.  
\item[\ma2] If $\rowsum_{\mindex(M)}(M) = 1$ and $\mindex(M) = \dim(M)$ then let $\reduce(M)$ equal $M$ with row $\dim(M)$ and column $\dim(M)$ removed.
\item[\ma3] If $\rowsum_{\mindex(M)}(M) = 1$ and $\mindex(M) < \dim(M)$ then perform the following modifications to $M$ to form $\reduce(M)$.  For $j \in [1, \mindex(M))$, move the entry in position $(j, \mindex(M))$ to position $(j, \dim(M))$.  Then delete row $\mindex(M)$ and (the empty) column $\mindex(M)$.
\end{enumerate}
Now recursively define $M^{(n)} = M$, and $M^{(k)} = \reduce(M^{(k+1)})$ for $k \in [1,n)$.  Let $a_k = \mindex(M^{(k)})-1$ for $i \in [1,n]$ and define $\bma(M) = (a_1, \ldots, a_n)$, which is an ascent sequence~\cite{dp}.  Note that in following this recursive procedure, the sequence $(a_1, \ldots, a_n)$ is constructed from right-to-left.
\end{definition}

\begin{remark}
A bijection between ascent sequences and Fishburn matrices is given in \cite{cyz} which differs from ours only in the definition of \am3.  Since our results are confined to those ascent sequences where \am3 is never invoked, our results would work equally well using the bijection of Chen, Yan and Zhou.  
\end{remark}

\subsubsection{Ascent sequences and posets}

We next introduce the bijections between $\Asc_n$ and $\Posets_n$ from \cite{bcdk}.  Given a (2+2)-free poset $P$, recall that $\level(P)$ denotes the highest index of a level or, equivalently, $\level(P)+1$ is the  number of levels in $P$.  Let $\minlevel(P)$ denote the minimum index of a level that contains a maximal element.  Recall the sequence $L(P)$ of level sets defined in Subsection~\ref{sub:classicalsets}.  The definition of $\bap$ below is very similar in structure to Definition~\ref{def:bam} above of $\bam$.

\begin{definition}\label{def:bap}
Given $\A = (\a_1, \ldots, \a_n) \in \Asc_n$, we define $\bap(\A)$ recursively.  First, $\bap(\A^{(1)})$ 
is the poset consisting of a single element $p_1$.  Supposing $P^{(k)} = \bap(\A^{(k)})$ for some $k \in [1,n)$, we have the following three cases for defining $P^{(k+1)} = \bap(\A^{(k+1)})$.
\begin{enumerate}
\item[\ap1]  If $\a_{k+1} \in [0,\minlevel(P^{(k)})]$ then let $P^{(k+1)}$ be the result of adding to $P^{(k)}$ a new maximal element $p_{k+1}$ that covers the same elements as do the elements in $L_{a_{k+1}}(P^{(k)})$.
\item[\ap2] If $\a_{k+1}=1+\level(P^{(k)})$ then let $P^{(k+1)}$ be the result of adding to $P^{(k)}$ a new element $p_{k+1}$ covering all maximal elements of $P^{(k)}$. 
\item[\ap3] If $\a_{k+1} \in (\minlevel(P^{(k)}), \level(P^{(k)})]$ 
then let $P^{(k+1)}$ be the outcome of the following: to $P^{(k)}$, 
add a new element $p_{k+1}$ covering the same elements as the elements in $L_{a_{k+1}}(P^{(k)})$.  Let $\mathcal{M}$ be the set of maximal elements of $P^{(k)}$ lying at any level less than $\a_{k+1}$.  
%Add all relations $x \preceq y$ where $x \preceq z$ for some $z \in \mathcal{M}$ and $y \in L_{a_{k+1}}(P^{(k)}) \cup \cdots \cup L_{\level(P^{(k)})}(P^{(k)})$ not including the new element $p_{k+1}$.
Add all relations $x \preceq y$ where $x \in \mathcal{M}$ and $y$ is any element of $L_{a_{k+1}}(P^{(k)}) \cup \cdots \cup L_{\level(P^{(k)})}(P^{(k)})$; 
here we do not consider the new element $p_{k+1}$ to be an element of $L_{a_{k+1}}(P^{(k)})$.
\end{enumerate}
Then $\bap(\A) = \bap(\A^{(n)}) = P^{(n)}$.  
\end{definition}

\begin{example}\label{exa:bap}
Let $\A = (0,0,1,2,0,1)$ as in Figure~\ref{fig:examples}.  Certainly, $P^{(1)}$ is the one-element poset.  The recursive construction of $P = \bap(\A)$ appears in Figure~\ref{fig:bapexample}, where the dotted shapes depict the different levels.  The element $a_{k+1}$ appears above each arrow and the case name appears below each arrow. The labels $p_i$ are just for expository purposes and are not part of $\bap(\A)$.  In the final step, $\mathcal{M} = \{p_5\}$.  Note that the new element $p_6$ in the final step ends up on its own level, and this is true in general for applications of \ap3.
\end{example}

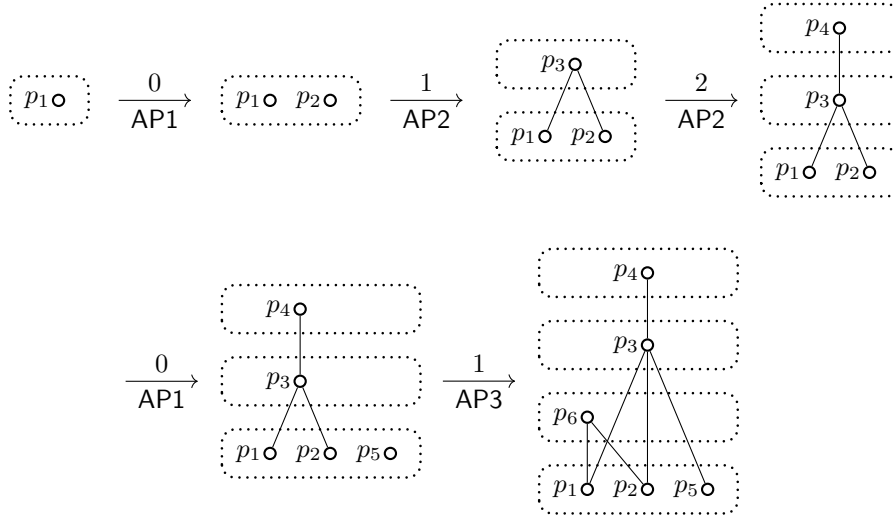
\begin{figure}
\begin{tikzpicture}[scale=0.8]
\tikzstyle{every node}=[circle, inner sep=1.5pt];
\begin{scope}
\draw[thick] (0,0) node[draw] (p1) {$$};
\draw (-0.35,0) node {$p_1$};
\draw [line width=1pt, line cap=round, dash pattern=on 0pt off 3pt, rounded corners=5pt] (-0.8,-0.4) rectangle (0.5,0.4);
\draw[->] (1,0) -- (2.2,0);
\draw (1.6,0.3) node {0};
\draw (1.6,-0.3) node {\ap1};
\end{scope}
\begin{scope}[xshift = 100]
\draw[thick] (0,0) node[draw] (p1) {$$};
\draw[thick] (1,0) node[draw] (p2) {$$};
\draw (-0.35,0) node {$p_1$};
\draw (0.65,0) node {$p_2$};
\draw [line width=1pt, line cap=round, dash pattern=on 0pt off 3pt, rounded corners=5pt] (-0.8,-0.4) rectangle (1.5,0.4);
\draw[->] (2,0) -- (3.2,0);
\draw (2.6,0.3) node {1};
\draw (2.6,-0.3) node {\ap2};
\end{scope}
\begin{scope}[xshift = 230]
\draw[thick] (0,-0.6) node[draw] (p1) {$$};
\draw[thick] (1,-0.6) node[draw] (p2) {$$};
\draw[thick] (0.5,0.6) node[draw] (p3) {$$};
\draw (-0.35,-0.6) node {$p_1$};
\draw (0.65,-0.6) node {$p_2$};
\draw (0.15,0.6) node {$p_3$};
\draw (p1) -- (p3) -- (p2);
\draw [line width=1pt, line cap=round, dash pattern=on 0pt off 3pt, rounded corners=5pt] (-0.8,-1) rectangle (1.5,-0.2);
\draw [line width=1pt, line cap=round, dash pattern=on 0pt off 3pt, rounded corners=5pt] (-0.8,0.2) rectangle (1.5,1);
\draw[->] (2,0) -- (3.2,0);
\draw (2.6,0.3) node {2};
\draw (2.6,-0.3) node {\ap2};
\end{scope}
\begin{scope}[xshift = 355, yshift = -17]
\draw[thick] (0,-0.6) node[draw] (p1) {$$};
\draw[thick] (1,-0.6) node[draw] (p2) {$$};
\draw[thick] (0.5,0.6) node[draw] (p3) {$$};
\draw[thick] (0.5,1.8) node[draw] (p4) {$$};
\draw (-0.35,-0.6) node {$p_1$};
\draw (0.65,-0.6) node {$p_2$};
\draw (0.15,0.6) node {$p_3$};
\draw (0.15,1.8) node {$p_4$};
\draw (p1) -- (p3) -- (p2);
\draw (p3) -- (p4);
\draw [line width=1pt, line cap=round, dash pattern=on 0pt off 3pt, rounded corners=5pt] (-0.8,-1) rectangle (1.5,-0.2);
\draw [line width=1pt, line cap=round, dash pattern=on 0pt off 3pt, rounded corners=5pt] (-0.8,0.2) rectangle (1.5,1);
\draw [line width=1pt, line cap=round, dash pattern=on 0pt off 3pt, rounded corners=5pt] (-0.8,1.4) rectangle (1.5,2.2);
\end{scope}
\begin{scope}[xshift = 100, yshift = -150]
\draw[thick] (0,-0.6) node[draw] (p1) {$$};
\draw[thick] (1,-0.6) node[draw] (p2) {$$};
\draw[thick] (0.5,0.6) node[draw] (p3) {$$};
\draw[thick] (0.5,1.8) node[draw] (p4) {$$};
\draw[thick] (2,-0.6) node[draw] (p5) {$$};
\draw (-0.35,-0.6) node {$p_1$};
\draw (0.65,-0.6) node {$p_2$};
\draw (0.15,0.6) node {$p_3$};
\draw (0.15,1.8) node {$p_4$};
\draw (1.65,-0.6) node {$p_5$};
\draw (p1) -- (p3) -- (p2);
\draw (p3) -- (p4);
\draw [line width=1pt, line cap=round, dash pattern=on 0pt off 3pt, rounded corners=5pt] (-0.8,-1) rectangle (2.5,-0.2);
\draw [line width=1pt, line cap=round, dash pattern=on 0pt off 3pt, rounded corners=5pt] (-0.8,0.2) rectangle (2.5,1);
\draw [line width=1pt, line cap=round, dash pattern=on 0pt off 3pt, rounded corners=5pt] (-0.8,1.4) rectangle (2.5,2.2);
\draw[->] (-2.4,0.6) -- (-1.2,0.6);
\draw (-1.8,0.9) node {0};
\draw (-1.8,0.3) node {\ap1};
\end{scope}
\begin{scope}[xshift = 250, yshift = -167]
\draw[thick] (0,-0.6) node[draw] (p1) {$$};
\draw[thick] (1,-0.6) node[draw] (p2) {$$};
\draw[thick] (1,1.8) node[draw] (p3) {$$};
\draw[thick] (1,3.0) node[draw] (p4) {$$};
\draw[thick] (2,-0.6) node[draw] (p5) {$$};
\draw[thick] (0,0.6) node[draw] (p6) {$$};
\draw (-0.35,-0.6) node {$p_1$};
\draw (0.65,-0.6) node {$p_2$};
\draw (0.65,1.8) node {$p_3$};
\draw (0.65,3.0) node {$p_4$};
\draw (1.65,-0.6) node {$p_5$};
\draw (-0.35,0.6) node {$p_6$};
\draw (p1) -- (p3) -- (p2);
\draw (p5) -- (p3) -- (p4);
\draw (p1) -- (p6) -- (p2);
\draw [line width=1pt, line cap=round, dash pattern=on 0pt off 3pt, rounded corners=5pt] (-0.8,-1) rectangle (2.5,-0.2);
\draw [line width=1pt, line cap=round, dash pattern=on 0pt off 3pt, rounded corners=5pt] (-0.8,0.2) rectangle (2.5,1);
\draw [line width=1pt, line cap=round, dash pattern=on 0pt off 3pt, rounded corners=5pt] (-0.8,1.4) rectangle (2.5,2.2);
\draw [line width=1pt, line cap=round, dash pattern=on 0pt off 3pt, rounded corners=5pt] (-0.8,2.6) rectangle (2.5,3.4);
\draw[->] (-2.4,1.2) -- (-1.2,1.2);
\draw (-1.8,1.5) node {1};
\draw (-1.8,0.9) node {\ap3};
\end{scope}
\end{tikzpicture}
\caption{The recursive construction of Example~\ref{exa:bap}.}
\label{fig:bapexample}
\end{figure}

To define the inverse $\bpa$, a key observation from Definition~\ref{def:bap} is that $p_{k+1}$ is always a maximal element of $P^{(k+1)}$ and, furthermore, $a_{k+1}$ will be exactly $\minlevel(P^{(k+1)})$.  Therefore, $\bpa$ will be a recursive map that obtains $P^{(k)}$ from $P^{(k+1)}$ by removing a maximal element at level $ \minlevel(P^{(k+1)})$ and records this level value as $a_{k+1}$.  Since all maximal elements at the same level in a (2+2)-free poset are order equivalent, it does not matter which element we remove.  Of course, care has to be taken to make sure we invert \ap3 correctly: see \pa3 below.

Recall that $D_i(P)$ denotes the strict downset of the elements at level $i$.  Again, the definition below parallels the corresponding one from Subsection~\ref{sub:bam}, i.e., Definition~\ref{def:bma}.
Obviously, Figure~\ref{fig:bapexample} in reverse serves as an example of Definition~\ref{def:bpa}.

\begin{definition}\label{def:bpa}
Given $P \in \Posets_n$, let $i = \minlevel(P)$ and define $\remove(P) \in \Posets_{n-1}$ in the following fashion.
\begin{enumerate}
\item[\pa1] If $|L_i| > 1$ then let $\remove(P)$ equal $P$ with a maximal element at level $i$ removed.  \item[\pa2] If $|L_i| = 1$ and $i = \level(P)$ then obtain $\remove(P)$ by deleting the unique element at level $i$ of $P$.
\item[\pa3] If $|L_i| = 1$ and $i < \level(P)$ then perform the following modifications to $P$ to form $\remove(P)$.   
Let $\mathcal{M} = D_{i+1}(P) \setminus D_i(P)$.  Make each element of $\mathcal{M}$ a maximal element by deleting all relations $x \prec y$ with $x \in \mathcal{M}$.  Then remove the unique element that was at level $i$ of $P$. 
\end{enumerate}
Now recursively define $P^{(n)} = P$, and $P^{(k)} = \remove(P^{(k+1)})$ for $k \in [1,n)$.  Let $a_k = \minlevel(P^{(k)})$ for $i \in [1,n]$ and define $\bpa(P) = (a_1, \ldots, a_n)$, which is an ascent sequence~\cite{bcdk}.  Note that in following this recursive procedure, the sequence $(a_1, \ldots, a_n)$ is constructed from right-to-left.
\end{definition}  

\subsubsection{Matrices and posets}  The bijections between matrices and posets are actually simpler than those involving ascent sequences.  To provide definitions for $\bmp$ and $\bpm$, it is is easier to work with labelled posets and remove the labels after the construction, as is done in~\cite{djk}. We begin with $\bpm$. 

\begin{definition}\label{def:bpm}
Let $P$ be a labelled (2+2)-free poset on the set $\{p_1,\ldots,p_n\}$.  The poset is uniquely specified by the sets $D(P)$ and $L(P)$ defined in Subsection~\ref{sub:classicalsets} since every element of $L_i$ has the set $D_i$ as its strict downset.

For $j \in [0, \level(P)]$, let $K_j(P)=D_{j+1}(P)\backslash D_j(P)$ where $D_{\level(P)+1} \coloneqq P$.
Let $M'$ be the matrix with entries $M'_{ij}=L_{i-1}(P)\cap K_{j-1}(P)$ for all $i,j \in [1,\level(P)+1]$, and define $\bpm(P)$ to be the matrix whose $(i,j)$ entry is $|M'_{ij}|$.
\end{definition}

\begin{example}
Let $P$ be the (2+2)-free poset from the top-left of Figure~\ref{fig:examples} where
\begin{align*}
D(P) &= (\emptyset,\{p_1,p_2\},\{p_1,p_2,p_5\},\{p_1,p_2,p_3,p_5\}) \\
L(P) &= (\{p_1,p_2,p_5\},\{p_6\},\{p_3\},\{p_4\})
\end{align*}
We have $\level(P)=3$ and 
$$K(P)=(K_0(P),K_1(P),K_2(P),K_3(P))=(\{p_1,p_2\},\{p_5\},\{p_3\},\{p_4,p_6\}).$$
From this
\begin{equation}\label{equ:setmatrix}
M' = \begin{pmatrix}
\{p_1,p_2\} & \{p_5\} & \emptyset & \emptyset \\
\emptyset & \emptyset & \emptyset & \{p_6\} \\
\emptyset & \emptyset & \{p_3\} & \emptyset \\
\emptyset & \emptyset & \emptyset & \{p_4\}  
\end{pmatrix}
\end{equation}
and so $\bpm(P)$ is the matrix $M$ in the bottom-left of Figure~\ref{fig:examples}.
\end{example}

To define $\bmp$ by inverting $\bpm$, the first step is to construct the set-valued matrix as in \eqref{equ:setmatrix}, but it is not clear which of the entries $\{p_1, \ldots, p_n\}$ should go where.  The key insight is that it doesn't really matter because $\bmp(M)$ will ultimately be an \emph{unlabelled} poset.  

\begin{definition}\label{def:bmp}
Let $M \in \Matrices_n$ and let $P'=\{p_1,\ldots,p_n\}$.
Replace each of the zero entries in $M$ with the empty set, and replace any non-zero integer $a$ in $M$ with a subset of $P'$ of size $a$ while ensuring there are no duplicated elements in the matrix.
Given  $p_i \in P'$, let $c(p_i)$ and $r(p_i)$ be the column  and row indices of $p_i$\,.
Let $P=(P',\preceq)$ be the poset whereby $p_i \prec p_j$ if and only if $c(p_i)< r(p_j)$.
Finally let $\bmp(M)$ be the unlabelled version of the poset $P$.\qed
\end{definition}

\begin{example}
Again using $M$ from Figure~\ref{fig:examples}, suppose we assign $\{p_1, \ldots, p_6\}$ as follows:
\[
M' = \begin{pmatrix} \{p_1,p_2\} & \{p_3\} & \emptyset & \emptyset \\
\emptyset & \emptyset & \emptyset & \{p_4\} \\
\emptyset & \emptyset & \{p_5\} & \emptyset \\
\emptyset & \emptyset & \emptyset & \{p_6\}\end{pmatrix}.
\]
This gives the strict order relations $p_1,p_2 \prec p_4,p_5,p_6$ and $p_3 \prec p_5,p_6$ and $p_5 \prec p_6$.
This is the same poset as in the top left of Figure~\ref{fig:examples} once all labels are removed.
\end{example}

\subsubsection{Ascent sequences to permutations}

For an ascent sequence $\A$, we follow~\cite{bcdk}, which used a technique systemised in~\cite{Wes}, to recursively construct a permutation $\bas(\A) \in \sn(2|3\overline{1})$. 

Given a  $2|3\overline{1}$-avoiding permutation $\pi \in \mathcal{S}_{n-1}$, we wish to add the entry $n$ to $\pi$ and get another $2|3\overline{1}$-avoiding permutation.  With this in mind, we say that the site between positions $i$ and $i+1$ in $\pi$ is \emph{active} if $\pi_i=1$ or $\pi_i-1$ is to the left of $\pi_i$.  The sites immediately before position 1 and immediately after position $n-1$ are always active.  We leave it as an exercise to check that the active sites are exactly those where inserting $n$ into $\pi$ will result in an element of $\sn(2|3\overline{1})$.  We label the active sites as subscripts from left to right beginning with 0.  For example, the permutation 521634 becomes $_0521_16_23_34_4$.

Recall that for an ascent sequence $\A = (a_1, \ldots, a_n)$, the prefix $(a_1, \ldots, a_k)$ is denoted $\A^{(k)}$.

\begin{definition}\label{def:bas}
 Given $\A = (a_1, \ldots, a_n) \in \Asc_n$, we define $\bas(\A)$ recursively.  First, $\bas(\A^{(1)}) = 1$, a permutation of one element.  For $k \in (1,n]$, we let $\bas(\A^{(k)})$ be $\bas(\A^{(k-1)})$ with $k$ inserted at the active site labelled $a_k$.  Then $\bas(\A)$ is defined to be $\bas(\A^{(n)})$.
\end{definition}

We leave it as an exercise to check that there will indeed exist an active site labelled $a_k$.  

\begin{example}
Let $\A = (0,0,1,2,0,1)$ as in Figure~\ref{fig:examples}.  We get the following sequence of permutations, labelled according to their active sites, with the arrows labelled by $a_k$:
\[
_01_1\ \xrightarrow{\;0\;}\ _021_1\ \xrightarrow{\;1\;}\ _021_13_2\ \xrightarrow{\;2\;}\ _021_13_24_3\ \xrightarrow{\;0\;}\ _0521_13_24_3\ \xrightarrow{\;1\;}\ _0521_16_23_34_4\ ,
\]
consistent with Figure~\ref{fig:examples}

\end{example}

\subsubsection{Relationships among the bijections}

Since the bijections between matrices and posets are very different from those involving ascent sequences, it is not at all clear that Figure~\ref{fig:summary} is a commutative diagram at the level of $\Posets$, $\Asc$ and $\Matrices$.  Since our maps are all bijections, this is equivalent to showing that $\bpm = \bam \circ \bpa$, the truth of which is given as a remark in \cite{djk}.  The equivalent statement $\bpm \circ \bap = \bam$ can be checked using a careful induction argument.  The crux of the argument is that the ways in which \ap1, \ap2 and \ap3 from Definition~\ref{def:bap} change the matrix $\bpm(P^{(k)})$ to $\bpm(P^{(k+1)})$ match exactly with the effects of \am1, \am2 and \am3 respectively from Definition~\ref{def:bam}.  In Theorem~\ref{thm:ascent_flip}, we will use the equivalent statement that $\bpa = \bma \circ \bpm$.

We will also need a bijection $\bps$ (see Subsection~\ref{sub:perm_dual}), and we will define it by $\bps = \bas \circ \bpa$.  A direct definition of $\bps$ that bypasses ascent sequences appears in \cite[Sub.~4.2]{bcdk} but we choose not to include that background here since we do not make heavy use of $\bps$.  
 
%%%%%%%%%%%%%%%%%%%%%%%%%%%%%%%
%%%%%%%%%%%%%%%%%%%%%%%%%%%%%%%
\section{Restricted sets}
%%%%%%%%%%%%%%%%%%%%%%%%%%%%%%%
%%%%%%%%%%%%%%%%%%%%%%%%%%%%%%%

As we saw in the previous section, the bijection $\bam$ (resp.\ $\bap$) becomes more complicated when the case \am3 (resp.\ \ap3) arises.  In this section, we introduce the subset of $\Asc$ for which these cases never arise, denoted $\RAsc$, where the letter $\R$ stands for ``restricted.''  There are two reasons for restricting our attention.  
First, as we will show, the images of this subset under our various bijections have nice definitions, and some of these images are combinatorially significant. Secondly, the bijections' behaviour becomes more tractable, allowing us to obtain results for these restricted sets, as we will see in particular in Section~\ref{sec:duality}.

We first must determine the conditions on an ascent sequence $\A$ that cause these two troublesome cases to arise.  Referring to Definition~\ref{def:bam}, we need to avoid the situation when
\begin{equation}\label{equ:rasc_matrices}
\a_{k+1} \in [\mindex(M^{(k)}),\dim(M^{(k)})). 
\end{equation}
Let us translate this expression into properties of $\A$.
Looking at all three cases in this definition, we see that $\mindex(M^{(k+1)})$ is nothing more than $a_{k+1}+1$, so $\mindex(M^{(k)}) = a_k+1$.  In addition, this observation implies that \am1 corresponds to the case when $a_k$ is not an ascent in $\A$.  Moreover, observe that when \am1 applies, the dimension of the matrix does not change, but it increases by 1 for each application of \am2 or \am3.  Therefore, $\dim(M^{(k)})$ is exactly 
$\asc(\a_1,\ldots,\a_{k})+1$.  Thus the condition \eqref{equ:rasc_matrices} that invokes \am3 when constructing $M^{(k+1)}$ is equivalent to
\begin{equation}\label{equ:am3}
\a_{k+1} \in (a_k, \asc(\a_1,\ldots,\a_{k})].
\end{equation}
A similar analysis of Definition~\ref{def:bap} shows that~\eqref{equ:am3} is also exactly the condition that causes \ap3 to be used.  Thus we are led to the following definition of a subset of $\Asc$, which was introduced in \cite{bcdk} under the name ``self-modified'' ascent sequences.

\begin{definition}
Let $\RAsc_n$ be the subset of $\Asc$ consisting of those ascent sequences $\A=(\a_1,\ldots,\a_n)$ such that
\begin{equation}\label{equ:rasc_original}
\a_k \in [0,\a_{k-1}] \cup \{1+\asc(\a_1,\ldots,\a_{k-1})\} \mbox{\ \ for all $k>1$}.
\end{equation}
\end{definition}

In other words, if $\a_k$ is larger than $a_{k-1}$, then $\a_k$ must be the largest it can be under the conditions on an ascent sequence.  
For example, $(0,1,0,2)$ is in $\RAsc$ whereas $(0,1,0,1)$ is not.  
Taking these ideas a step further, we get the following equivalent definition of $\RAsc$, which will be useful later.

\begin{lemma}
$\RAsc_n$ is the subset of $\Asc$ consisting of those ascent sequences $\A=(\a_1,\ldots,\a_n)$ such that
\begin{equation}\label{equ:rasc}
\a_k \in [0,\a_{k-1}] \cup \{1+\max(\a_1,\ldots,\a_{k-1})\} \mbox{\ \ for all $k>1$}.
\end{equation}
\end{lemma}
\begin{proof}
We argue by induction on the number of ascents.  If the first ascent is at position $i_1$ so that $0 = a_{i_1} < a_{i_1+1}$ then we have 
\[
a_{i_1+1} = 1 + \asc(\a_1,\ldots,\a_{i_1}) = 1 + \max(\a_1,\ldots,\a_{i_1}) = 1.
\]
With the $j$th ascent at position $i_j$, we assume
\[
a_{i_j+1} = 1 + \asc(\a_1,\ldots,\a_{i_j}) = 1 + \max(\a_1,\ldots,\a_{i_j}).
\]
Thus at the $(j+1)$st ascent, we have (note the subtle difference below between the indices $i_{j+1}$ and $i_j+1$)
\begin{eqnarray*}
a_{i_{j+1}+1}  & = & 1 + \asc(\a_1,\ldots,\a_{i_{j+1}}) \\
& =  & 1 + \asc(\a_1,\ldots,\a_{i_j+1}) \\\
& = & 1 + \asc(\a_1,\ldots,\a_{i_j}) + 1 \\
& = & 1 + \max(\a_1,\ldots,\a_{i_j}) + 1 \\
& = & \max(\a_1,\ldots,\a_{i_j+1}) + 1 \\
& = & \max(\a_1,\ldots,\a_{i_{j+1}})+1.
\end{eqnarray*}
\end{proof}

The rest of this section will be working towards proving Corollary~\ref{cor:restricted_bijections}, which states that, under the classical bijections defined in the previous section, $\RAsc_n$ maps to the sets we now define.

\begin{definition} \
\begin{itemize}
\item
Let $\RMatrices_n$ be the set of matrices in $\Matrices_n$ all of whose diagonal entries are positive.
\item 
Let $\RPosets_n$ be the set of those posets in $\Posets_n$ that have a chain of length $\level(P)$.  
\item 
Let $\RPerms_n  = \sn(3\overline{1}52\overline{4})$ (defined next).
\end{itemize}
\end{definition}

The set $\sn(3\overline{1}52\overline{4})$ is enumerated in \cite{Pud, Pud2} (see also \cite[\href{https://oeis.org/A098569}{A098569}]{oeis}) and appears in \cite{bcdk} as the image $\bas(\RAsc_n)$.  
A permutation $\pi$ is said to avoid the barred permutation $3\overline{1}52\overline{4}$ if every occurrence of the pattern 231 in $\pi$ plays the role of 352 in an occurrence of the pattern 31524 in $\pi$.  In other words, if we have $i<j<k$ with $\pi_k < \pi_i < \pi_j$\,, there must also exist $\ell$ and $m$ such that $i < \ell < j < k < m$ and $\pi_i \pi_\ell \pi_j \pi_k \pi_m$ is an occurrence of 31524.

Since a poset $P$ in $\Posets$ has $\level(P)+1$ levels, the maximal possible length of a chain is $\level(P)$, so $\RPosets_n$ consists of those posets that have a chain that contains an element from every level.    

The following observation will be crucial for several results stated in this paper.  It not only gives the image of an element of $\RMatrices$ under $\bma$ but, combined with Corollary~\ref{cor:restricted_bijections}(a), shows that all elements of $\RAsc$ take a particular form.  We use $i^j$ to denote a sequence of $j$ copies of $i$.  From this point on, we will denote the entry in row $i$ and column $j$ of the matrix $M$ by $m_{ij}$ or $m_{i,j}$

\begin{lemma}\label{rmattorasc}
Suppose that $M \in \RMatrices$ with $\dim(M)=d$.  Then 
\begin{equation}\label{equ:asc2matrix}
\bma(M) = (0^{m_{11}},1^{m_{22}},0^{m_{12}}, 2^{m_{33}},1^{m_{23}},0^{m_{13}},\ldots,(d-1)^{m_{dd}},(d-2)^{m_{d-1\,d}},\ldots,0^{m_{1d}}).
\end{equation}
\end{lemma}

\begin{proof}
This follows from the definition of $\bma$ by considering how $\A = \bma(M)$ is constructed from right-to-left in Definition~\ref{def:bma}.  Since the diagonal entries of $M$ are all positive, we never have to apply \ma3.  
\end{proof} 

\begin{example}
The matrix $M = \left(\begin{smallmatrix} 2&0&1 \\ 0&1&0 \\ 0&0&1\end{smallmatrix}\right)$ is in bijection with the ascent sequence $\A = (0,0,1,2,0)$.
\end{example}

The next lemma shows the essential condition for the desired bijection between $\RAsc_n$ and $\RMatrices_n$.  

\begin{prop}\label{pro:diag_zeros}
Let $\A=(\a_1,\ldots,\a_n) \in \Asc_n$ and $M=\bam(\A)$ with $\dim(M)=d$.
There will be a zero on the diagonal of $M$ if and only if there exists $i\in [1,n-1]$ such that $\a_i<\a_{i+1}\leq \asc(\a_1,\ldots,\a_i)$.
\end{prop}

\begin{proof}
To show this we will consider Definition~\ref{def:bam} and how the matrix entries change during 
construction with respect to the rules \am1, \am2 and \am3.

Suppose that there exists $i\in [1,n-1]$ with $\a_i<\a_{i+1}\leq \asc(\a_1,\ldots,\a_i)$.  
Recall from the second paragraph of this section that $\a_i = \mindex(M^{(i)})-1$ and $\asc(a_1,\ldots,a_i) = \dim(M^{(i)})-1$.
So when constructing $M^{(i+1)}$ from $M^{(i)}$, \am3 applies, and 
$M^{(i+1)}_{1+\a_{i+1},1+\a_{i+1}} =0$. This entry is not in the rightmost column of the matrix 
and cannot therefore be increased by any subsequent applications of \am1 or \am2. 
Depending on the subsequent values in the ascent sequence, if \am3 is used then the 0 at position $(1+\a_{i+1},1+\a_{i+1})$ may be permuted 
amongst the diagonal entries but will never again be in the rightmost column, and therefore never again accessible to change.

For the converse, suppose that for all $i\in [1,n-1]$,
\begin{equation}\label{equ:rasc_inequalities}
\a_{i+1} \in [0,\a_i] \cup \{1+\asc(\a_1,\ldots,\a_i)\}.
\end{equation}
The procedure for constructing $M$ will begin with $M^{(1)} = (1)$, a 1-by-1 matrix.  Because of~\eqref{equ:rasc_inequalities}, only \am1 and \am2 will be needed to construct the subsequent $M^{(i)}$.  Since \am1 and \am2 both preserve the property of all diagonal entries being strictly positive, $M$ will also have that desired property.
\end{proof}

We next provide a similarly essential ingredient for the desired bijection between $\RPosets$ and $\RMatrices$.

\begin{prop}\label{pro:rposets_to_rmatrices}
Let $P \in \Posets_n$ and let $M = \bpm(P)$ with $dim(M) = \level(P)+1$.  All the entries on the diagonal of $M$ will be non-zero if and only if $P$ has a chain of length $\level(P)$.
\end{prop}

\begin{proof}
Consider the construction of $M$ from $P$ as given in Definition~\ref{def:bpm}, especially the intermediate matrix $M'$ given by 
\[
M'_{ij} = L_{i-1} \cap (D_{j}\backslash D_{j-1})
\]
for all $i,j \in [1,\level(P)+1]$.  Suppose $m_{ii} \neq 0$ for all $i \in [1,\level(P)+1]$. We have that $m_{ii} \neq 0$ if and only if $M'_{ii} \neq \emptyset$, which
is equivalent to 
\begin{equation}\label{equ:nonempty_intersection}
L_{i-1} \cap (D_{i}\backslash D_{i-1}) \neq \emptyset.
\end{equation}
In other words, there exists at least one element $w_i \in L_{i-1} \cap (D_{i}\backslash D_{i-1})$ for all $1\leq i \leq \level(P)+1$.  Therefore,
\begin{equation}\label{equ:chain}
w_1 \prec_P w_2 \prec_P \cdots \prec_P w_{\level(P)},
\end{equation}
and so $P$ has a chain of length $\level(P)$.  

To show that if $P$ has a chain of length $\level(P)$ then $M$ has only non-zero diagonal entries, it suffices to show that the condition in~\eqref{equ:chain} implies that in~\eqref{equ:nonempty_intersection} for all $i$.  So consider the construction of $P$ from $M$, as given in Definition~\ref{def:bmp}.  A chain in $P$ of length $\level(P)$ must arise from a sequence of $\level(P)+1$ nonempty entries in $M'$, each strictly southeast of the previous one.  
But  since $\dim(M')=\level(P)+1$ by definition of $\bpm$, these nonempty entires in $M'$ must all be along the diagonal, yielding~\eqref{equ:nonempty_intersection} for all $i$.
\end{proof}

Combining the previous two propositions gives the main result of this section, which shows that $\RAsc$, $\RMatrices$, $\RPosets$ and $\RPerms$ are all in bijection, and the bijections we need are exactly the restricted versions of the classical ones.  Part (c) of the theorem was already proved as~\cite[Prop.~10]{bcdk}

\begin{corollary}\label{cor:restricted_bijections}  The classical bijections among $\Asc_n$, $\Matrices_n$, $\Posets_n$ and $\Perms_n$ restrict to bijections among $\RAsc_n$, $\RMatrices_n$, $\RPosets_n$ and $\RPerms_n$, Specifically, 
\begin{enumerate}
\item $\bam$ maps $\RAsc_n$ bijectively to $\RMatrices_n$;
\item $\bpm$ maps $\RPosets_n$ bijectively to $\RMatrices_n$;
\item \cite{bcdk} $\bas$ maps $\RAsc_n$ bijectively to $\RPerms_n$.
\end{enumerate}
\end{corollary}

\begin{proof}
Part (a) follows from Prop.~\ref{pro:diag_zeros}, the definitions of $\RAsc_n$ and $\RMatrices_n$,  and the fact that $\bam: \Asc_n \to \Matrices_n$ is a bijection.  

Similarly, (b) follows from Prop.~\ref{pro:rposets_to_rmatrices}, the definitions of $\RPosets_n$ and $\RMatrices_n$,  and the fact that $\bpm: \Posets_n \to \Matrices_n$ is a bijection. 
\end{proof}

%%%%%%%%%%%%%%%%%%%%%%%%%%%%%%%
%%%%%%%%%%%%%%%%%%%%%%%%%%%%%%%
\section{Poset duality under the bijections}\label{sec:duality}
%%%%%%%%%%%%%%%%%%%%%%%%%%%%%%%
%%%%%%%%%%%%%%%%%%%%%%%%%%%%%%%

If a poset $P$ is (2+2)-free, then it is clear that the dual poset $P^*$ obtained by reversing all its inequalities is also (2+2)-free.  An open question in~\cite{dp} asks how $\A$ is related to $\A^*$, where $\A$ and $\A^*$ are the ascent sequences corresponding to $P$ and $P^*$ respectively.  While this question appears intractable for general (2+2)-free posets, in this section we answer it for $\RPosets$.  In addition, we extend the answer to give the corresponding notion of duality for $\RPerms$.  Combined with the duality result for $\RMatrices$ given by~\cite[Theorem~10]{djk}, we get a complete understanding of how poset duality acts on our four $\R$-families according to our bijections.  In fact, one major motivation for our restriction to the $\R$-families is their amenability to adopting an analogue of poset duality.  

In view of Corollary~\ref{cor:restricted_bijections}, we can abuse notation by using the same $f$ notation for our bijections even though our domains will now be $\R$-families as opposed to the domains of $\Asc$, $\Matrices$, $\Posets$ and $\Perms$ that we had before. 

\begin{definition}
Let $f:\RPosets \to \Structure$ be a bijection where $\Structure$ is a collection of objects.
Given $P \in \RPosets$ with $f(P)=s$, we write $s^{*}$ for the unique object $f(P^{*})$, and we call $s^{*}$ the \emph{dual} of $s$ according to $f$.
\end{definition}

\begin{example}\label{exa:matrix_dual}
As a first example, we consider the dual of an element $M = (m_{ij})$ of $\RMatrices$ according to the bijection $\bpm$.  Define $\flip{M}$ to be the reflection of $M$ through its antidiagonal, i.e., if $\dim(M)=d$, then $\flip{M}_{ij} = m_{d+1-j,\, d+1-i}$.  Observe that $M \in \RMatrices_n$ if and only if $\flip{M} \in \RMatrices_n$.  Theorem 10 from~\cite{djk} states that $M^* = \flip{M}$.
\end{example}

%%%%%%%%%%%%%%%%%%%%%%%%%%%%%%%
\subsection{Duality for ascent sequences}\label{sub:dual_ascent_sequences}
%%%%%%%%%%%%%%%%%%%%%%%%%%%%%%%

We will use Example~\ref{exa:matrix_dual} as a basis for determining the dual of an element of $\RAsc$ according to $\bpa$.  %Using $\flip{M}$ to determine $\A^*$ is valid because $\bpa = \bma \circ \bpm$.

\begin{theorem}\label{thm:ascent_flip}
Let $\A \in \RAsc_n$.  By Lemma~\ref{rmattorasc}, we have
\begin{equation}\label{equ:asc2matrix2}
\A = (0^{m_{11}},1^{m_{22}},0^{m_{12}}, 2^{m_{33}},1^{m_{23}},0^{m_{13}},\ldots,(d-1)^{m_{dd}},(d-2)^{m_{d-1\,d}},\ldots,0^{m_{1d}})
\end{equation}
where $M=\bam(\A)$. 
The dual ascent sequence $\A^*$ according to $\bpa$ is given by
\begin{multline}\label{equ:asc_dual}
\A^*= (0^{m_{dd}},1^{m_{d-1\, d-1}},0^{m_{d-1\, d}}, 2^{m_{d-2\, d-2}},1^{m_{d-2\,d-1}},0^{m_{d-2\, d}},\\
\ldots,(d-1)^{m_{11}},(d-2)^{m_{12}},\ldots,0^{m_{1d}}).
\end{multline}
\end{theorem}

\begin{proof}
Let $P = \bap(\A)$.  We first observe that, since $\bpa = \bma \circ \bpm$, we get  $\A^* = \bpa(P^*) = \bma(\flip{M})$.  Considering the definition of $\flip{M}$ and applying Lemma~\ref{rmattorasc} yields the result.
\end{proof}

\begin{example}
If $\A = (0,0,1,2,0)$, then $M = \left(\begin{smallmatrix} 2&0&1 \\ 0&1&0 \\ 0&0&1\end{smallmatrix}\right)$.  Thus $\flip{M} = \left(\begin{smallmatrix} 1&0&1 \\ 0&1&0 \\ 0&0&2\end{smallmatrix}\right)$, and so $\A^* = (0,1,2,2,0)$.
\end{example}

There is an alternative way to construct $\A^*$ that avoids the need to write it out in the form \eqref{equ:asc2matrix2}.   This alternative uses a map $\D$ that takes \emph{any} sequence of numbers as its input and returns an ascent sequence $\D(\A)$, which we call the \emph{panorama} of $\A$.  First, we need a preliminary definition.

\begin{definition}
Given a sequence $(\a_1, \a_2, \ldots, \a_k)$ of real numbers, the \emph{view} $v_i$ of the element $\a_i$ is defined in the following recursive fashion:

\begin{enumerate} 
\renewcommand{\theenumi}{\alph{enumi}}
\item if $j$ is the minimum index greater than $i$ such that $\a_j > \a_i$, then define $v_i = v_j + 1$;
\item if no such $j$ exists then $v_i \coloneqq 0$.  
\end{enumerate}
\end{definition}
We can think of the view $v_i$ as counting left-to-right maxima starting at the entry $\a_i$.  For example, the sequence $(0,1,2,0,4,3,1,2)$ has $(3,2,1,1,0,0,1,0)$ as its sequence of views.

\begin{definition}
Let $\A$ be a sequence of real numbers.  We construct the \emph{panorama sequence} $\D(\A)$ of $\A$ in the following manner.  Suppose the maximum value \linebreak of $\A$ occurs in positions $i_1, i_2, \ldots, i_j$.  Begin $\D(\A)$ with the sequence of views \linebreak $(v_{i_1}, v_{i_2}, \ldots, v_{i_j})$.  Now continue this process with the next highest value in $\A$, concatenating the corresponding view values onto the right end of $\D(\A)$.  Repeat this process until the view of every element of $\A$ has been added to $\D(\A)$.
\end{definition}

\begin{example}\label{exa:panorama}
For $\A = (0,1,2,0,4,3,1,2)$ as above, $\D(\A) = (0,0,1,0,2,1,3,1)$.  
\end{example}

As promised, we get the following alternative construction of $\A^*$.  

\begin{prop}\label{pro:dual_alternative}
Let $\A \in \RAsc$ and let $\A^*$ be its dual according to $\bpa$.  Then $\A^* = \D(\A)$.  
\end{prop}

\begin{proof}
Consider $\D(\A)$ with $\A$ as given in \eqref{equ:asc2matrix}.  
By Corollary~\ref{cor:restricted_bijections}(a) and by definition of $\RMatrices_n$, we know that the diagonal entries $m_{jj}$ are all positive.  
With this observation in mind, we see that the left-to-right maxima that determine the view of the element $\a_i$ 
come from those entries in $\A$ of the form $j^{m_{j+1\, j+1}}$ with $j > i$ that appear to the right of $\a_i$.  
It follows that the panorama $\D(\A)$ is exactly the sequence given in \eqref{equ:asc_dual}, as required.
\end{proof}

Observe that $\D(\A)$ in Example~\ref{exa:panorama} is an ascent sequence, while Proposition~\ref{pro:dual_alternative} implies the same is true for all $\A \in \RAsc$; this is no coincidence.  In fact, one can show that $\D(\A)$ is an ascent sequence for \emph{any} sequence of real numbers $\A$.  One method of proof is to show $\D(\A)$ has a stronger property, 
originally defined in \cite{hu}, which we recall now.  
Another reason for introducing this natural stronger property is that it will be useful in the next section. 

\begin{definition}
A \emph{restricted growth function (RGF)} is a sequence $\A$ of non-negative integers such that each $j>0$ that appears in $\A$ is preceded by an appearance of $j-1$. Equivalently, for all $j>0$ that appear in $\A$, the first appearance of $j$ is preceded by an appearance of every $i$ satisfying $0 \leq i < j$.  
\end{definition}

For example, $(0,1,0,2,1,3)$ is an RGF whereas $(0,1,0,1,3,2)$ is not.  

\begin{lemma}\label{lem:rgf_containment}
$\RAsc \subseteq \RGF \subseteq \Asc$ and both containments are strict.
\end{lemma}

\begin{proof}
Both containments are trivial for sequences of length $n=1$, so assume $n>1$ and  
let $\A = (\a_1, \ldots, \a_n) \in \RAsc_n$.  By \eqref{equ:rasc}, the first time $j>0$ appears, it must take the form $1+\max(\a_1, \ldots, \a_{i-1})$ for some $i$.  Thus $j-1$ appears before $j$, and $\A$ is an RGF.  

Next let $\A = (\a_1, \ldots, \a_n)$ be an RGF, and let $\eell(j)$ denote the position in $\A$ of the leftmost appearance of the number $j$, when such a position exists.  Note that to prove that $\A \in \Asc$, it suffices to show that $\a_{\eell(j)} \leq 1 + \asc(\a_1, \ldots, \a_{\eell(j)-1})$ for each $j$.  By definition of RGFs, $\eell(0) < \eell(1) < \cdots < \eell(m)$, where $m$ is the maximum value in $\A$.  Hence position $\eell(i)-1$ is always an ascent for $1\leq i \leq m$. Thus $\asc(\a_1, \ldots, \a_{\eell(j)-1}) \geq j-1$, from which we conclude $\a_{\eell(j)} = j \leq 1 + \asc(\a_1, \ldots, \a_{\eell(j)-1})$, as required.  

To see the strict containments, the shortest examples are 
$(0,1,0,1) \in \RGF \setminus \RAsc$ while $(0,1,0,1,3) \in \Asc \setminus \RGF$.
\end{proof}

\begin{rem}
Clearly, generalised ballot sequences (also known as Yamanouchi words) are RGFs, and so we have (slightly) generalised the first author's result~\cite{ballot} that generalised ballot sequences are ascent sequences.  
\end{rem}

\begin{rem}
We leave it as an exercise for the reader to justify our earlier assertion that $\D(\A)$ is an RGF and hence an ascent sequence for any sequence of real numbers $\A$.  (In fact, one can constructively prove the converse that every RGF is a panorama sequence.)  A harder exercise is to show there is another way in which $\D(\A)$ is ``nicer'' than $\A$: if $\A$ is an RGF, then $\D(\A) \in \RAsc$.  Consequently, $\D^2(\A) \in \RAsc$ for any sequence $\A$ of real numbers.  Moreover, this latter fact combined with Proposition~\ref{pro:dual_alternative} tells us that $\D^2(\A) = \A$ if and only if $\A \in \RAsc$.  This last observation gives an alternative definition of $\RAsc$.
\end{rem}

As we mentioned, we have not been able to answer in full the open question from \cite{dp} by determining $\A^*$ for all $\A \in \Asc$.  However, we assert that $\RAsc$ is in some sense the largest subset that behaves nicely with respect to poset duality.  Indeed, since $(\A^*)^* = \A$ by definition, the fact that $\D^2(\A) = \A$ if and only if $\A \in \RAsc$ tells us that $\A^*$ can be as simple as $\D(\A)$ if and only if $\A \in \RAsc$.

%%%%%%%%%%%%%%%%%%%%%%%%%%%%%%%
\subsection{Duality for permutations}\label{sub:perm_dual}
%%%%%%%%%%%%%%%%%%%%%%%%%%%%%%%

For our one remaining notion of duality, we use the definition of the dual of an ascent sequence to determine the dual $\pi^*$ of an element $\pi$ of $\RPerms = \sn(3\overline{1}52\overline{4})$ according to $\bps \coloneqq \bas \circ \bpa$.

Our definition of $\pi^*$ requires the use of what are perhaps the three best known involutions on a permutation $\pi = (\pi_1, \ldots, \pi_n)$: its inverse $\pi^{-1}$, its \emph{reverse} $\reverse(\pi) \coloneqq (\pi_n, \ldots, \pi_1)$, and its \emph{complement} $\complement(\pi) \coloneqq (n+1-\pi_1, \ldots, n+1-\pi_n)$.  Together, these involutions allow us to state the result with the most technical proof of this paper.

\begin{theorem}
Let $\pi \in \sn(3\overline{1}52\overline{4}) = \RPerms$. The dual permutation according to $\bps$ is given by $\pi^{*} =(\complement(\reverse(\pi)))^{-1}$.
\end{theorem}

\begin{proof}
Let $\A \in \RAsc_n$ and let $\pi = \bas(\A) \in  \sn(3\overline{1}52\overline{4})$.  
See Example~\ref{exa:dual_perm} below for an example pertinent to the key elements of this proof; we will use the symbol  $\checkmark$ in this proof to denote paragraphs or single statements that are demonstrated in the example.
Theorem~\ref{thm:ascent_flip} tells us how to construct $\A^{*}$ from $\A$:
if 
\begin{equation}\label{equ:rasc2}
\A = (0^{m_{11}},1^{m_{22}},0^{m_{12}}, 2^{m_{33}},1^{m_{23}},0^{m_{13}},\ldots,(d-1)^{m_{dd}},(d-2)^{m_{d-1\,d}},\ldots,0^{m_{1d}})
\end{equation}
then $\A^{*}$ is given by
\begin{multline*} \A^{*}= (0^{m_{dd}},1^{m_{d-1\, d-1}},0^{m_{d-1\, d}}, 2^{m_{d-2\, d-2}},1^{m_{d-2\,d-1}},0^{m_{d-2\, d}}, \\
\ldots,(d-1)^{m_{11}},(d-2)^{m_{12}},\ldots,0^{m_{1d}}).
\end{multline*}
Here, $\A = \bma(M)$ and $\A^* = \bma(\flip{M})$ as shown in the proof of Theorem~\ref{thm:ascent_flip}.  $\checkmark$.

In Definition~\ref{def:bas}, we saw a recursive definition of $\bas(\A)$.  However, as shown in \cite[Cor.~9]{bcdk}, when $\A \in \RAsc$, we have the following equivalent but simpler definition.
Let $\A =(\A_1,\ldots,\A_n) \in \RAsc_n$ and let $k$ be the largest value in this sequence.
Let $W_i(\A)$ be the list of all positions $j\in [1,n]$ such that $\A_j=i$ written in decreasing order.
Define 
\begin{equation}\label{equ:wdef}
\bas(\A) =  W_0(\A) W_1(\A) \ldots W_k(\A)=\bigoplus_{i=0}^k W_i(\A)
\end{equation}
to be the concatenation of these lists $\checkmark$.

The proof will rely on several different ways to sum the $m_{ab}$ values which we now identify.
The sums use three different ways to traverse the upper-triangular entries of $(m_{ij})$.  \begin{description}
\item[T traversal] $(1,1) \to (2,2) \to (1,2) \to (3,3) \to (2,3) \to (1,3) \to \ldots \to (2,d) \to (1,d)$.
\item[R traversal] $(1,d) \to (2,d) \to \ldots \to (d,d) \to (1,d-1) \to \ldots \to (d-1,d-1) \to \ldots \to (1,1)$.
\item[S traversal] $(1,d) \to (1,d-1) \to \ldots \to (1,1) \to (2,d) \to \ldots (2,2) \to \ldots \to (d-1,d) \to (d-1,d-1) \to (d,d)$.
\end{description}
Let $T_{ij}$ be the sum of $m_{ab}$ using T traversal until we reach the pair $(a,b)=(i,j)$.
Let $R_{ij}$ be the sum of $m_{ab}$ using R traversal until we reach the pair $(a,b)=(i,j)$.
Let $S_{ij}$ be the sum of $m_{ab}$ using S traversal until we reach the pair $(a,b)=(i,j)$.
The quantities $T'_{ij}, R'_{ij}$, and $S'_{ij}$ are those where one sums $m'_{ab}$'s in place of $m_{ab}$'s, where the two are related via $m'_{ij}=m_{d+1-j,\,d+1-i}$.  
In other words, the primed versions of the sums are obtained by traversing in the prescribed orders but over the entries of $\flip{M}$ instead of $M$ $\checkmark$.
Notice that we have the following identities:  
\begin{align}
T_{ij}+R_{ij} = n+m_{ij} \label{tandr}\\
T'_{ij}+R'_{ij} = n +m'_{ij} \label{tandrdual}\\
S'_{ij} = R_{d+1-j,\,d+1-i}\ . \label{sdualandr}
\end{align}

Let us use the notation $[x]_a$ for the list $(x,x-1,\ldots,x-a+1)$, and define $[x]_0$ to be the empty list.  Applying $\bas(\A)$ we find that $W_{i-1}(\A) = [T_{id}]_{m_{id}} \cdots [T_{ii}]_{m_{ii}}$.  This is due to the T traversal matching the order of the powers $m_{ij}$ appearing in~\eqref{equ:rasc2}.  Thus
\begin{equation}\label{equ:pi_from_T}
\pi = \bas(\A) =  \bigoplus_{i=1}^d \bigoplus_{j=d}^{i} [T_{ij}]_{m_{ij}}
\end{equation}
where $\oplus$ denotes left-to-right concatenation of sequences, and where the index $j$ in the inner concatenation runs from $d$ down to $i$ $\checkmark$.
Similarly, working with $\flip{M}$ instead of $M$, we have
\begin{equation}\label{equ:pi_prime_from_T_prime}
\pi^{*} = \bas(\A^{*}) =  \bigoplus_{i=1}^d \bigoplus_{j=d}^{i} [T'_{ij}]_{m'_{ij}}\ \checkmark
\end{equation}
We now have bona fide expressions for both $\pi$ and $\pi^{*}$. It remains to show that they satisfy the equation stated in the theorem.

The value $\pi(a)$ is obtained in the following way: 
\begin{equation}\label{equ:pi_a}
\pi(a) = T_{ij}+S_{ij}-m_{ij}+1-a
\end{equation}
where $(i,j)$ is the unique pair such that $a\in (S_{ij}-m_{ij},S_{ij}]$ $\checkmark$.  
Indeed, to determine $\pi(a)$, we first have to determine which sequence $[T_{ij}]_{m_{ij}}$ from~\eqref{equ:pi_from_T} will include the $a$th entry of $\pi$.  The order from left-to-right in which the $[T_{ij}]_{m_{ij}}$ appear matches the S traversal, which is why we pick $(i,j)$ in the stated way.  Next we need the $\pi(a)$ values for $a\in (S_{ij}-m_{ij},S_{ij}]$ to give the sequence 
$[T_{ij}]_{m_{ij}}$.  
Indeed, the smallest such $a$ is $S_{ij}-m_{ij}+1$, which gives $\pi(a) = T_{ij}$ in~\eqref{equ:pi_a}. As $a$ increases all the way to $S_{ij}$, we get all the $\pi(a)$ values decreasing all the way to $T_{ij} - m_{ij}+1$, as required.
This argument also makes it clear that the pair $(i,j)$ is also the unique pair such that $b =\pi(a) \in (T_{ij}-m_{ij},T_{ij}]$ $\checkmark$.

The statement of the theorem is equivalent to showing $\pi^*(a) = b$ if and only if $\pi(n+1-b) = n+1-a$.
Suppose that $\pi^*(a) = b$. 
Then we have $\pi^*(a) = T'_{ij} +S'_{ij}-m_{ij}' +1-a$ where $(i,j)$ is the unique pair such that $a \in (S'_{ij}-m_{ij}', S'_{ij}]$. 
Equivalently $(i,j)$ is the unique pair such that $b=\pi^*(a) \in (T_{ij}'-m_{ij}',T_{ij}']$.
The rest of this proof will start with the statement that
\begin{equation}\label{equ:pi_star}
T'_{ij}+S'_{ij} -m'_{ij} +1-a = b,
\end{equation}
where $(i,j)$ is the unique pair such that $b \in (T_{ij}'-m_{ij}',T_{ij}']$ and consist of manipulations to remove all primed terms with a view to recovering a statement that is equivalent to $\pi(n+1-b) = n+1-a$.

From~\eqref{tandr}--\eqref{sdualandr}, we have $T'_{ij}-m'_{ij} = n-S_{d+1-j,\,d+1-i}$\,, and 
\[
S'_{ij} = R_{d+1-j,\,d+1-i} = n+m_{d+1-j,\,d+1-i}-T_{d+1-j,\,d+1-i}\ .
\]
The above equality~\eqref{equ:pi_star} is therefore equivalent to
$$n-S_{d+1-j,\,d+1-i} +(n+m_{d+1-j,\,d+1-i}-T_{d+1-j,\,d+1-i})+1-b  =a.$$
Subtract both sides of this equation from $n+1$ to yield
$$T_{d+1-j,\,d+1-i} + S_{d+1-j,\,d+1-i} - m_{d+1-j,\,d+1-i} + 1- (n+1-b) = n+1-a$$
where, again, $(i,j)$ is the unique pair such that $b \in (T_{ij}'-m_{ij}',T_{ij}']$.
Notice that 
\begin{eqnarray*}
& & b \in (T_{ij}'-m_{ij}',T_{ij}'] \\
%b \in (T_{ij}'-m_{ij}',T_{ij}']
& \iff &
T_{ij}'-m_{ij}' < b \leq T_{ij}'\\
& \iff&
n+1-T_{ij}'+m_{ij}' > n+1- b \geq n+1-T_{ij}'\\
& \iff&
n-T_{ij}'+m_{ij}' \geq  n+1- b > n-T_{ij}'\\
& \iff&
R_{ij}' \geq  n+1- b > R_{ij}'- m_{ij}'\\
& \iff&
S_{d+1-j,\,d+1-i} \geq  n+1- b > S_{d+1-j,\,d+1-i}- m_{d+1-j,\,d+1-i}\\
& \iff&
n+1-b \in (S_{d+1-j,\,d+1-i}- m_{d+1-j,\,d+1-i}, S_{d+1-j,\,d+1-i}].
\end{eqnarray*}

Thus we now have that
$$T_{d+1-j,\,d+1-i} + S_{d+1-j,\,d+1-i} - m_{d+1-j,\,d+1-i} + 1  - (n+1-b) = n+1-a$$
where $(i,j)$ is the unique pair such that
\[
n+1-b \in (S_{d+1-j,\,d+1-i}- m_{d+1-j,\,d+1-i}, S_{d+1-j,\,d+1-i}]. 
\]
Changing the variables to $I \coloneqq d+1-j$ and $J\coloneqq d+1-i$, we get 
$$T_{I,J}  + S_{I,J} - m_{I,J} + 1 - (n+1-b) = n+1-a$$
where $(I,J)$ is the unique pair such that $n+1-b \in (S_{I,J} - m_{I,J}, S_{I,J}]$.  
By comparing with \eqref{equ:pi_a}, we see that we have arrived at a statement that is equivalent to $\pi(n+1-b)=n+1-a$, as required.
\end{proof}

\begin{example}\label{exa:dual_perm}
Let
\[
\A = (0,0,0,1,1,0,2,2,2,2,0,3,1,1)
\]
and so 
\[
\pi = \bas(\A) = (11,6,3,2,1\,|\,14,13,5,4\,|\,10,9,8,7\,|\,12)
\]
where the vertical bars separate each $W_i(\A)$ from $W_{i+1}(\A)$.  We have 
$\A = \bma(M)$ where 
\[
M = \begin{pmatrix} 3 & 1 & 1 & 0 \\ 0 & 2 & 0 & 2 \\ 0 & 0 & 4 & 0 \\ 0 & 0 & 0 & 1 \end{pmatrix}
\]
by~\eqref{equ:rasc2}.  Thus 
\[
\flip{M} = \begin{pmatrix} 1 & 0 & 2 & 0 \\ 0 & 4 & 0 & 1 \\ 0 & 0 & 2 & 1 \\ 0 & 0 & 0 & 3 \end{pmatrix}
\]
from which we can read off
\[
\A^* = (0, 1, 1, 1, 1, 2, 2, 0, 0, 3, 3, 3, 2, 1), 
\]
yielding
\[
\pi^* = (9, 8, 1 \,|\, 14, 5, 4, 3, 2 \,|\, 13, 7, 6 \,|\, 12, 11, 10)
\]
which does indeed equal $(\complement(\reverse(\pi)))^{-1}$.  

From here on, we give examples of other items appearing in the proof.  First, we have $T_{33} = 3+2+1+4=10$, $R_{33} = 2+1+1+4=8$, $S_{33} = 1+1+3+2+2+4=13$, $T'_{33} = 1+4+2=7$, $R'_{33} = 1+1+3+2+2=9$, and $S'_{33} = 2+1+1+4+1+2=11$.  Next, consistent with~\eqref{equ:pi_from_T} and~\eqref{equ:pi_prime_from_T_prime}, we have
\[
\pi = [14]_0 [11]_1 [6]_1 [3]_3\ |\ [14]_2 [10]_0 [5]_2\ |\ [12]_0 [10]_4\ |\ [12]_1.
\]
and
\[
\pi^* = [14]_0 [9]_2 [5]_0 [1]_1\ |\ [14]_1 [7]_0 [5]_4\ |\ [13]_1 [7]_2\ |\ [12]_3.
\]

Finally, let us give some examples of~\eqref{equ:pi_a}, starting with $a=10$.  To determine $\pi(10)$ we find that $10 \in (S_{33}-m_{33}, S_{33}] = (9,13]$.  We thus obtain $\pi(10) = T_{33} + S_{33} - m_{33} + 1 -10 = 10 + 13 - 4 +1 -10 = 10$.  As $a$ increases to $11, 12, 13$, the only term that changes on the right-hand side of~\eqref{equ:pi_a} is $a$, and we get $9,8,7$ for the respectively values of $\pi(a)$.  Notice that the four values we obtained here for $\pi(a)$ are exactly the elements of $[T_{33}]_{m_{33}}$.
\end{example}

%%%%%%%%%%%%%%%%%%%%%%%%%%%%%%%
%%%%%%%%%%%%%%%%%%%%%%%%%%%%%%%
\section{A Catalan restriction and series-parallel posets}\label{sec:catalan}
%%%%%%%%%%%%%%%%%%%%%%%%%%%%%%%
%%%%%%%%%%%%%%%%%%%%%%%%%%%%%%%

In this section, we consider subsets of the $\R$-families whose cardinalities are given by the Catalan numbers; we thus use the $\C$ prefix for naming these subsets.  The study of these subsets is also motivated by the results of applying our bijections to these subsets, which allow us to draw connections between some natural families: pattern-avoiding ascent sequences, pattern-avoiding permutations, and series-parallel (2+2)-free posets.  Pattern avoidance in ascent sequences is studied from an enumerative perspective in \cite{ds}.

A sequence is said to be $abab$-avoiding if it is both $0101$-avoiding and $1010$-avoiding.

\begin{prop}\label{pro:as_catalan}
Let $\A$ be a sequence of non-negative integers.  The following are equivalent:
\begin{enumerate}
\item $\A$ is a $101$-avoiding ascent sequence;
\item $\A$ is an $abab$-avoiding ascent sequence;
\item $\A$ is a $0101$-avoiding ascent sequence;
\item $\A$ is an $abab$-avoiding RGF.
\end{enumerate}
Consequently, the number of sequences of length $n$ satisfying these conditions is the Catalan number $C_n$.  
\end{prop}

\begin{proof}
That (a) implies (b) is immediate, as is the implication from (b) to (c).  That (a) and (c) are equivalent is~\cite[Theorem~2.5]{ds}.  Lemma~\ref{lem:rgf_containment} gives that  (d) implies (b), while~\cite[Lemma~2.4]{ds} gives that both (a) or (c) imply (d).  

The last assertion is shown in~\cite[Theorem~2.5]{ds} and also follows from the fact that the number of sequences $\A$ of length $n$ satisfying (d) is shown to be $C_n$ in~\cite[Exer.~88]{catalan}.  
\end{proof}

\begin{definition}
Let $\CAsc_n$ denote the set of those sequences of length $n$ consisting of non-negative integers that satisfy the conditions of Proposition~\ref{pro:as_catalan}. 
\end{definition}

As an example, of the 15 ascent sequences of length 4, the only one containing $abab$ is $0101$.  Thus $14 = C_4 = |\CAsc_4|$.

%\begin{remark}
%Under the bijection that sends a set partition $S$ to an RGF $\A$, we see that $S$ will be a non-crossing partition if and only if $\A$ is an $abab$-avoiding RGF~\cite[Theorem~2.6]{ds}.  The non-crossing partitions are a well-known Catalan family~\cite[Exer.~159]{catalan}.
%\end{remark}

It is certainly not the case that $\CAsc_n = \RAsc_n$ since $01021 \in \RAsc_n \setminus \CAsc_n$.  However, we do have the following relationship.

\begin{prop}\label{pro:RcontainsS}
$\CAsc_n \subseteq \RAsc_n$.
\end{prop}

\begin{proof}
Let $\A = (\a_1, \ldots, \a_n)$ be an ascent sequence that is not in $\RAsc_n$.  We will show that $\A$ violates one of the characterizations of $\CAsc_n$ from Proposition~\ref{pro:as_catalan}.  By definition of $\RAsc_n$, there exists $i$ such that  $\a_{i-1} < \a_i < 1 + \asc(\a_1, \ldots, \a_{i-1})$.    If $\a_i = \a_j$ for some $j < i-1$, then $\a_j \a_{i-1} \a_i$ is an occurrence of 101 in $\A$, violating Proposition~\ref{pro:as_catalan}(a).   If $\a_j > \a_i$ for some $j < i-1$, then the RGF property implied by Proposition~\ref{pro:as_catalan}(d) would again imply the appearance of the subsequence  $\a_i \a_{i-1} \a_i$  in $\A$.  Thus we can assume that $\a_j <\a_i$ for all $j$ with $1 \leq j \leq i-1$.  

If $\A \in \CAsc_n$, we know $\A$ is an RGF, so the first appearance $\a_j$ of any \linebreak non-zero element of $(\a_1, \ldots, \a_{i-1})$ results in position $j-1$ contributing $+1$ to $\asc(\a_1, \ldots, \a_{i-1})$.  
Because $\a_j < \a_i$, the combined contribution of these first appearances to $\asc(\a_1, \ldots, \a_{i-1})$ is at most $\a_i-1$.
Since $\a_i -1 <  \asc(\a_1, \ldots, \a_{i-1})$, there must then be a value $\a_j$ that results in at least two ascents, i.e., there must be exist $j$ and $k$ such that $1 < j < k \leq i-1$ with $\a_j=\a_k$ such that both $j-1$ and $k-1$ are ascents.  Consequently, $\a_j \a_{k-1} \a_k$ is an occurrence of 101 in $\A$, again violating Proposition~\ref{pro:as_catalan}(a).
\end{proof}

We next turn to determining the image of $\CAsc_n$ under our bijections.  We begin with posets since this case is perhaps the most interesting.  Recall from Figure~\ref{fig:summary} that $\CPosets_n$ denotes the series-parallel posets in $\Posets_n$.  In other words, $\CPosets_n$ consists of those posets with $n$ elements that are both $(2+2)$-free and N-free.

\begin{prop}\label{pro:casc}
$\bap: \CAsc_n \to \CPosets_n$, and $\bap$ is a bijection.
\end{prop}
 
\begin{proof}
Since $\bap:\Asc_n\to\Posets_n$ is a bijection and since $\CAsc_n$ and $\CPosets_n$ both have size $C_n$ by Proposition~\ref{pro:as_catalan} and~\cite[Exer.~182]{catalan}, it suffices to prove that $\bap(\A) \in \CPosets_n$ for any $\A = (\a_1, \ldots, \a_n) \in \CAsc_n$.    So suppose $\bap(\A) = P$ contains an N.  Since $\A \in \RAsc_n$ by Proposition~\ref{pro:RcontainsS}, we know that only rules \ap1 and \ap2 from Definition~\ref{def:bap} are used in constructing $P$. Let $i, j, k, \eell$ be distinct positions in $\A$ that result (under $\bap$) in the elements $e, f, g, h$, respectively, of $P$ which form an N as, for example, in Figure~\ref{fig:N}. 
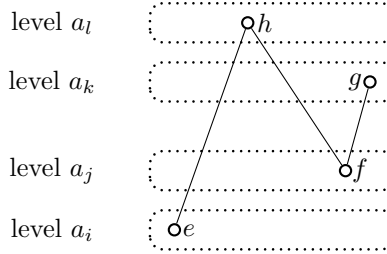
\begin{figure}
\begin{center}
\begin{tikzpicture}[scale=0.65]
\tikzstyle{every node}=[circle, inner sep=1.5pt];
\begin{scope}
\draw[thick] (0.5,0) node[draw] (e) {$$};
\draw (0.8,0) node {$e$};
\draw (-2,0) node {level $\a_i$};
\draw [line width=1pt, line cap=round, dash pattern=on 0pt off 3pt, rounded corners=5pt] (0,-0.4) rectangle (5,0.4);
\end{scope}
\begin{scope}[yshift=8ex]
\draw[thick] (4,0) node[draw] (f) {$$};
\draw (4.3,0) node {$f$};
\draw (-2,0) node {level $\a_j$};
\draw [line width=1pt, line cap=round, dash pattern=on 0pt off 3pt, rounded corners=5pt] (0,-0.4) rectangle (5,0.4);
\end{scope}
\begin{scope}[yshift=20ex]
\draw[thick] (4.5,0) node[draw] (g) {$$};
\draw (4.2,0) node {$g$};
\draw (-2,0) node {level $\a_k$};
\draw [line width=1pt, line cap=round, dash pattern=on 0pt off 3pt, rounded corners=5pt] (0,-0.4) rectangle (5,0.4);
\end{scope}
\begin{scope}[yshift=28ex]
\draw[thick] (2,0) node[draw] (h) {$$};
\draw (2.35,0) node {$h$};
\draw (-2,0) node {level $\a_\eell$};
\draw [line width=1pt, line cap=round, dash pattern=on 0pt off 3pt, rounded corners=5pt] (0,-0.4) rectangle (5,0.4);
\end{scope}
\draw (e) -- (h) -- (f) -- (g);
\end{tikzpicture}
\caption{For ease of reference of notation, we give one possible configuration of the N-subposet of $P$ from the proof of Proposition~\ref{pro:casc}.  This is just one of two possible configurations, since there is no mathematical reason why level $\a_i$ should appear below level $\a_j$.}
\label{fig:N}
\end{center}
\end{figure}
Since $\A \in \CAsc_n \subseteq \RAsc_n$, applying~\cite[Lem.~7]{bcdk} tells us that $e, f, g, h$ appear at levels $\a_i$, $\a_j$, $\a_k$, $\a_\eell$ respectively of $P$.  
Comparing strict downsets, we see that $\a_\eell > \a_k > \a_j, \a_i$.   There are two cases to consider depending on the relative values of the positions $k$ and $\eell$.  

We first consider the case $\eell > k$.  Since $\A$ is an RGF by Proposition~\ref{pro:as_catalan}, in constructing $P$, the first time an element was added at level $\a_\eell$ it must have been added as a top element using the rule \ap2.  But since $h \not>_P g$, there must be some element $h'$ at level $\a_\eell$ that was added before $g$ (and so also before $h$).  Thus we have $\a_\eell \a_k \a_\eell$ appearing in $\A$, contradicting Proposition~\ref{pro:as_catalan}(a).

It remains to consider the case $\eell < k$.  With the same reasoning as the previous paragraph, there exists an element $g'$ at level $\a_k$ that was added as a top element.  Since $e \not\in D_g$ and $g$ and $g'$ are at the same level, we know $e \not\in D_{g'}$.  Thus $g'$ was added to $P$ before $e$.  On the other hand, since $h >_P e$, we know $h$ was added to $P$ after $e$, and hence so was $g$.  Thus we have $\a_k \a_i \a_k$ appearing in $\A$, again contradicting Proposition~\ref{pro:as_catalan}(a).
\end{proof}

We follow Jel\'inek \cite{JelFishburn} to define the family $\CMatrices_n$.
% that we will show will equal $\bam(\CAsc_n)$, we need to define a new property of matrices.

\begin{definition}
An SE-pair of a matrix $M \in \Matrices$ is a pair of non-zero entries $m_{ij}$ and $m_{i'j'}$ such that $i < i'$, $j<j'$ and $i' \leq j$.  We say that $M$ is \emph{SE-free} if it contains no SE-pair.
\end{definition}

In \cite[Lem.~1.2]{JelFishburn}, Jel\'inek shows that $P \in \Posets$ is series-parallel if and only if $\bpm(P)$ is SE-free.

\begin{definition}
Define $\CMatrices_n$ to be the subset of $\Matrices_n$ consisting of those elements that are SE-free.
\end{definition}

%\begin{theorem}
%%$\bam(\CAsc_n) = \CMatrices_n$.
%$\bam$ maps $\CAsc_n$ bijectively to $\CMatrices_n$.
%\end{theorem}
%
%\begin{proof}
%Corollary~\ref{cor:restricted_bijections}(a) tells us that $\bam: \RAsc_n \to \RMatrices$ is a bijection.  Abusing notation, from Lemma~\ref{lem:alpattern} we further know that $\bam:  \RAsc_n\backslash \CAsc_n \to\RMatrices_n\backslash \CMatrices_n$ is a bijection.  Then the map from the difference between the two domains to the difference between the two ranges, $\bam: \CAsc_n \mapsto \CMatrices$, is also a bijection.
%\end{proof}

Finally, we determine the $\C$-family for permutations.  The answer is quite appealing, namely $\CPerms = \sn(231)$.

\begin{remark}
At this point, it is worth clarifying the containment relations of Figure~\ref{fig:summary}.  We already showed in Proposition~\ref{pro:RcontainsS} that if $\A \in \Asc$, then $\A$ being 101-avoiding automatically implies that $\A$ is self-modified.  That the analogous implications hold for the other three $\C$-families follows from Theorem~\ref{cor:c}, but can also be checked directly as follows:
\begin{itemize}
\item  If $\pi \in \sn(231)$ then $\pi \in \sn(3\overline{1}52\overline{4})$ by definition of $\sn(3\overline{1}52\overline{4})$.
\item  If $M \in \Matrices$ is a $k$-by-$k$ matrix that is SE-free, then it must have only positive diagonal entries.  To see this, suppose $m_{ii} = 0$ for some $i$.  By definition of $\Matrices$, we have $2 \leq i \leq k-1$, and there must exist an SE-pair formed by $m_{ai}$ and $m_{ib}$ where $1\leq a < i$ and $i < b \leq k$. 
\item We leave it as an exercise to show that if $P \in \Posets$ is series-parallel, then $P$ has a chain of length $\level(P)$.
%Solution: find lowest level $i$ such that there is no chain saturated chain from level 0 to some element at $p$ at level $i$.  By equality of strict down-sets, not just one but every element $q$ at level $i-1$ is the top of a saturated chain from level 0.  Pick any such $q$.  Since $p$ is at a higher level than $q$, there must exist an element $r$ lying at level $j < i-1$ such that $r<p$ but $r \not< q$.  Then $p, q, r, s$ form an N or a (2+2) for some $s$ in the down-set of $q$.
\end{itemize}
\end{remark}

\begin{prop}
$\bas$ maps $\CAsc_n$ bijectively to $\CPerms$.
\end{prop}

\begin{proof}
The definition of $\CAsc_n$ and the fact that $\bas$ is a bijection tells us that $|\bas(\CAsc_n)| = C_n$.
As is well known, $C_n$ is also the cardinality of $\CPerms$.  Thus it suffices to show that $\CPerms\subseteq \bas(\CAsc_n)$.  Crucial to our proof is that, from Propositions~\ref{pro:as_catalan}(a) and \ref{pro:RcontainsS}, we have $\CAsc_n = \RAsc_n(101)$.

Corollary~\ref{cor:restricted_bijections}(c) and the definition of $\RPerms$ give $\bas(\RAsc_n) = \RPerms \supseteq \CPerms$.  Thus if $\bas(\A) \in \CPerms$, then $\A \in \RAsc_n$.  We wish to show the more refined fact that $\A \in \CAsc_n$, i.e., that $\A$ avoids $101$.  So suppose that $\A$ contains a $101$ pattern as $(\a_i,\a_j,\a_k)=(d,c,d)$.  Then in $\bas(\A)$, using the definition of $\bas(\A)$ given in \eqref{equ:wdef}, $j$ will be in $W_c(\A)$ which is to the left of $W_d(\A)$. In $W_d(\A)$, $k$ will be to the left of $i$ since $k>i$.
Thus $(j,k,i)$ will be a 231 pattern in $\bas(\A)$, which is the desired contradiction.
\end{proof}

We now have all the necessary ingredients to compile the following ``Catalan'' refinement of Corollary~\ref{cor:restricted_bijections}

\begin{corollary}\label{cor:c} The classical bijections among $\Asc_n$, $\Matrices_n$, $\Posets_n$ and $\Perms$ restrict to bijections among $\CAsc_n$, $\CMatrices_n$, $\CPosets_n$ and $\CPerms$. Specifically, 
\begin{enumerate} 
\item $\bap$ maps $\CAsc_n$ bijectively to $\CPosets_n$;
\item \cite{JelFishburn} $\bpm$ maps $\CPosets_n$ bijectively to $\CMatrices_n$;
\item $\bas$ maps $\CAsc_n$ bijectively to $\CPerms$.
\end{enumerate}
\end{corollary}

%%%%%%%%%%%%%%%%%%%%%%%%%%%%%%%
%%%%%%%%%%%%%%%%%%%%%%%%%%%%%%%
\section{Open problems}\label{sec:conclusion}
%%%%%%%%%%%%%%%%%%%%%%%%%%%%%%%
%%%%%%%%%%%%%%%%%%%%%%%%%%%%%%%

Comparing our bijections, the definition of those between $\Matrices$ and $\Posets$ are perhaps the simplest, which helps with our understanding of that correspondence.  In particular, although we determined in Section~\ref{sec:duality} the effect of poset duality on just the $\R$-families, the result from \cite{djk} that $M^* = \flip{M}$ given in Example~\ref{exa:matrix_dual} holds for all $M \in \Matrices$.  At the other extreme, the bijections $\bam$, $\bap$, and their inverses are the most complicated, especially because of the cases \am3 and \ap3.  It is not surprising, therefore, that our open problems all involve ascent sequences. 

As we have already mentioned, we have not answered the following question of \cite{dp} for $\A \in \Asc \setminus \RAsc$.

\begin{problem}  
Given $P \in \Posets$ with $\bpa(P) = \A$, how do we obtain $\bpa(P^*)$ from $\A$?
\end{problem}

The following problem has already been considered in \cite{ky,kr}.

\begin{problem}\label{pro:semiorder}
Under $\bap$, which ascent sequences give rise to posets that are both (2+2)-free and (3+1)-free?
\end{problem}

Posets that are both (2+2)- and (3+1)-free are known as \emph{semiorders} or \emph{unit interval orders}, and it is well known that they are counted by the Catalan numbers.  Analogous to Problem~\ref{pro:semiorder}, a question already answered by \cite[Prop.~16]{djk} is to determine those $M \in \Matrices$ that map under $\bmp$ to semiorders.  The answer is those $M$ such that there does not exist $i > i'$ and $j<j'$ with $m_{ij}m_{i'j'}\neq0$.  For example, in Figure~\ref{fig:examples}, $m_{33}m_{24} \neq 0$ while $\{p_3, p_4, p_5, p_6\}$ form a (3+1) in $P$.

Obtaining sets in bijection with semiorders can be motivated by a famous problem in symmetric functions: a conjecture of Stanley and Stembridge \cite{ss} states that the chromatic symmetric functions of incomparability graphs of (3+1)-free posets are $e$-positive.  
Guay-Paquet \cite{Gua} has reduced this conjecture to the case of semiorders.  So perhaps an alternative characterization of semiorders would give insight into the conjecture.  For example, we have a characterization of semiorders in terms of matrices $M$ in the previous paragraph, and \cite[Lem.~15]{djk} tells us an easy way to identify which elements of $\bmp(M)$ are incomparable.  

Finally, Proposition~\ref{pro:casc} suggests the following modification of Problem~\ref{pro:semiorder}.

\begin{problem}
Under $\bap$, which ascent sequences give rise to posets that are (2+2)-free, (3+1)-free and N-free?
\end{problem}

\section*{acknowledgements}
We are grateful to the anonymous referees (both for the present paper and the FPSAC 2019 extended abstract) for their helpful suggestions.  We thank V\'it Jel\'inek for the reference for Collary~5.9(b).
This work was partially supported by an EPSRC grant (EP/M015874/1 to Mark Dukes) and by a grant from the Simons Foundation (\#245597 to Peter McNamara).

%%%%%%%%%%%%%%%%%%%%%%%%%%%%%%%
%%%%%%%%%%%%%%%%%%%%%%%%%%%%%%%

\end{document}